\documentclass[reqno]{amsart}
\usepackage{amscd}
\usepackage[dvips]{graphics}

\usepackage{graphicx,subfigure,float,caption2,color,epstopdf}

\def\beq{\begin{equation}}
\def\eeq{\end{equation}}
\def\ba{\begin{array}}
\def\ea{\end{array}}

\def \RpN{\mathbb{R}_+^{n+1}}

\def\cal{\mathcal}

\numberwithin{equation}{section}
\newenvironment{abs}{\textbf{Abstract}\mbox{  }}{ }
\newenvironment{key words}{\textbf{Keywords}\mbox{  }}{ }
\newtheorem{theorem}{Theorem}[section]
\newtheorem{definition}[theorem]{\textbf{Definition}}
\newtheorem{corollary}[theorem]{\textbf{Corollary}}
\newtheorem{proposition}[theorem]{\textbf{Proposition}}
\newtheorem{lemma}[theorem]{Lemma}

\renewenvironment{proof}{\noindent{\textbf{Proof.}}}{\hfill$\Box$}
\theoremstyle{remark}
\newtheorem{remark}[theorem]{\textbf{Remark}}
\theoremstyle{plain}

\begin{document}

\title{\textbf{Hardy-Sobolev inequalities with distance to the boundary weight functions}}
\author  {Lei Wang  and Meijun Zhu}
\address{ Lei Wang, Beijing International Center for Mathematical Research, Peking University, Beijing,  100871, P.R.China }
\email{leiwang@bicmr.pku.edu.cn}

\address{ Meijun Zhu, Department of Mathematics,
The University of Oklahoma, Norman, OK 73019, USA; Institute of Geometry and Physics, the University of Science and Technology of China, Hefei, Anhui 230026, China}

\email{mzhu@math.ou.edu}


%\date{11-19-2018}
\maketitle

% ----------------------------------------------------------------
\noindent
\begin{abs}
This is the first part of our research on certain sharp Hardy-Sobolev inequalities and the related elliptic equations. In this part we shall establish some sharp weighted Hardy-Sobolev inequalities whose weights are distance functions to the boundary.
%inequalities on the upper half space.
\end{abs}\\

%\begin{key words} Liouville theorem, Positive kernel, Extension operators, Fractional Laplacian, Moving sphere method
%\end{key words}\\
%\textbf{Mathematics Subject Classification(2000).}
%35BXX \indent
%---------------------------------------------------------------------------------
\section{\textbf{Introduction}\label{Section 1}}
%

%{\bf Started to  write it up  5-30-2020 in Norman}

%come back on April 22, 2021

\subsection{A brief review on the classical Hardy and Sobolev inequalities}

One of the simple, yet indispensable tools   in the study of modern nonlinear partial differential equations is the following Hardy inequality (initially discovered by Hardy, see, for example,  \cite{HLP34}):
\begin{align}\label{Hardy-n}
    \int_{\mathbb{R}^{n+1}}\frac{|u(x)|^p}{|x|^p}dx\leq (\frac{p}{n+1-p})^p \int_{\mathbb{R}^{n+1}}|\nabla u|^p dx,\ \ \forall u\in \cal D_0^{1,p}(\mathbb{R}^{n+1}),
\end{align}
where $1<p<n+1$ and $\cal D_0^{1,p}({\mathbb{R}^{n+1}})$ is the completion of $C^{\infty}_0(\mathbb{R}^{n+1})$ under the norm $(\int_{\mathbb{R}^{n+1}}|\nabla u|^pdx)^{\frac1p}$.

The importance of this inequality is at least twofold.  
First of all, the inequality has its own interest. The constant $(\frac{p}{n+1-p})^p$ in \eqref{Hardy-n} is sharp. However, since the inequality has the scaling invariant property, an extremal sequence for the sharp inequality may not strongly converge in a suitable function space to any function. In fact, it is well known that the equality in  \eqref{Hardy-n} does not hold for any nontrivial function in $\cal D_0^{1,p}(\mathbb{R}^{n+1})$. Inequality \eqref{Hardy-n} certainly holds on any bounded domain containing the origin  as an interior point  for functions vanishing outside the domain, with the same sharp constant as in $\mathbb{R}^{n+1}$.

Secondly,  the inequality plays an essential role in solving certain classical nonlinear PDEs. For example, with the help of Hardy inequality and H\"older inequality, Hardy and Littlewood first  obtained the Hardy-Littlewood inequality (\cite{HL30}), and later it was discovered that its sharp form (so-called Bliss Lemma, see \cite{Bliss30}) yields the following sharp Sobolev inequality (due to  Aubin \cite{Aubin76a} and 
Talenti \cite{Talenti76}):
\begin{equation}\label{sob-1-1}
(\int_{\mathbb{R}^{n+1}}|u|^{\frac{(n+1)p}{n+1-p}}dx)^{\frac{n+1-p}{n+1}}\le S_{n+1,p} \int_{\mathbb{R}^{n+1}}|\nabla u|^p dx,\ \ \forall u\in \cal D_0^{1,p}({\mathbb{R}^{n+1}}),
\end{equation}
where $1<p<n+1$ (well, the sharp Sobolev inequality for $p=2$ in three dimension seems to be  derived first  by Rosen \cite{Rosen71}). Inequality \eqref{sob-1-1} in turn leads to the resolution of the famous Yamabe problem (see, for example,  Lee and Parker \cite{LP87}). Note that the above Sobolev inequality with a non-sharp constant can also be derived from  the  Gagliardo-Nirenberg inequality in Gagliardo \cite{Ga59} or Nirenberg \cite{Ni59}. For $p=2$, inequality \eqref{sob-1-1} holds on any domain in $\mathbb{R}^{n+1}$  for functions vanishing outside the domain, with the same sharp constant as in $\mathbb{R}^{n+1}$, but the sharp constant is never achieved unless the domain is the whole space $\mathbb{R}^{n+1}$ (due to the famous  Liouville Theorem of Gidas, Ni and Nirenberg \cite{GNN79}, or  the stronger Liouville Theorem (without decay assumption) of Caffarelli, Gidas and Spruck \cite{CGS89}).

Using the interpolation between  Sobolev inequality ($s=0$) and  Hardy inequality ($s=p$), one can easily obtain the following 
Hardy-Sobolev  inequality:
\begin{align}\label{HS-1}
     \big(\int_{\mathbb{R}^{n+1}}\frac{|u(x)|^{\frac{(n+1-s)p}{n+1-p}}}{|x|^s}dx\big)^{\frac{n+1-p}{n+1-s}}\leq C \int_{\mathbb{R}^{n+1}}|\nabla u|^p dx,\ \ \forall u\in \cal D_0^{1,p}({\mathbb{R}^{n+1}}),
\end{align}
where $1<p<n+1$ and $0\leq s\leq p$.  The sharp constants and extremal functions of \eqref{HS-1} for $0<s<p$ 
can also be obtained from Bliss Lemma. Some interesting studies of related inequalities on a bounded domain %$\Omega(\subset\mathbb{R}^n)$ for $p=2$
can be found, for example,  in Ghoussoub and Yuan\cite{
GY00}, Ghoussoub and Roberts\cite{GR06} and references therein.

\subsection{Hardy-Sobolev inequality with distance to the boundary weight functions}
For any bounded domain $\Omega\subset\mathbb{R}^{n+1}$ (for $n\geq 1$) with Lipschitz boundary, let $\delta=
\delta_\Omega(x)=\text{dist}(x,\partial \Omega)$ be the distance (to the boundary)  function for $x\in\Omega$.  Note that $\frac{\int_\Omega |\nabla u|^p dx}{\int_\Omega |\frac u{\delta}|^pdx}$ has the similar scaling invariant property to   $\frac{\int_\Omega |\nabla u|^p dx}{\int_\Omega\frac{|u(x)|^p}{|x-x_o|^p }dx}$ for any $x_o \in \Omega$. In fact, it is not surprise that we have another type of Hardy inequality, which asserts (see, for example, Opic and  Kufner \cite{OK90}, or see our direct proof in Section 2.1.2): for $p>1$,  there is a positive constant $C=C(n,p,\Omega)$, % \leq (\frac{p}{p-1})^p$, 
such that 
\begin{equation}\label{hardy-1-1}
\int_\Omega |\frac u{\delta}|^pdx \le C \int_\Omega |\nabla u|^p dx,\ \ \forall u\in W_0^{1,p}(\Omega).
\end{equation}

Similarly, using the interpolation between  Sobolev inequality 
and  Hardy inequality \eqref{hardy-1-1}, we can easily obtain a Hardy-Sobolev type inequality   similar  to  \eqref{HS-1} for $1<p<n+1$. In fact, we have a more general inequality as follows:

\begin{theorem}\label{thm1-1} Let $\Omega$ be a bounded domain with Lipschitz boundary in $\mathbb{R}^{n+1}$.  Assume $p\in (1, n+1]$ and $\beta$ satisfies
\begin{equation}\label{beta-p}
    \begin{cases}
       0\leq \beta \leq \frac{p(n+1)}{n+1-p},& \text{ if } p<n+1,\\
       \beta\geq 0, & \text{ if }p=n+1.
    \end{cases}
\end{equation}
Then there is a positive constant $C=C(n, p, \beta,\Omega)$, such that for all $u \in W^{1,p}_0(\Omega)$,
\begin{align}\label{p-main-1}
  \big( \int_{\Omega}\delta^{-p+\frac{n+1-p}{n+1}\beta}|u|^{p+\frac{p\beta}{n+1}}dx\big)^{\frac
   {n+1}{n+\beta+1}}\leq C\int_{\Omega} |\nabla u|^p dx.
\end{align}
\end{theorem}

    We will provide another approach to prove Theorem \ref{thm1-1}. We first prove the inequality on the upper half space $\mathbb{R}^{n+1}_+:=\{(y, t) \, :  \ y\in \mathbb{R}^n, \, t>0\} $ (see the next theorem), and then inequality \eqref{p-main-1} on a bounded domain can be derived by using the covering method via the  partition of unity. The main idea is similar to the approach for the proof of the $\varepsilon$-level sharp Sobolev type inequalities on compact manifolds, see, for example, the work by Aubin \cite{Aubin76b}, Hebey and Vaugon \cite{HV96},  Li and Zhu \cite{LZ97}, etc.
The advantage of this approach is twofold. First, we can obtain the above inequality for the case $p=n+1$, which can not be obtained via the interpolation method. Second, for $p=2$ and some specific $\beta$, we can obtain the sharp constants and extremal functions for the  inequalities on $\mathbb{R}^{n+1}_+$. In our later work in progress, such sharp constants and extremal functions play the key role in the study of sharp inequalities on a bounded domain for $p=2$.

Denote $\cal D_{0,0}^{1,p}(\mathbb{R}^{n+1}_+ )$ 
as the completion of $C^{\infty}_0(\mathbb{R}^{n+1}_+ )$ under the norm $(\int_{\mathbb{R}^{n+1}_+ }|\nabla u|^pdx)^{\frac1p}$. For $1<p\leq n+1$, it is easy to check that $$ \cal{D}^{1,p}_{0,0}(\mathbb{R}^{n+1}_+)=\{u\in W^{1,1}_{loc}(\mathbb{R}^{n+1}_+): u\in L^{p}(\mathbb{R}^{n+1}_+, t^{-p}dydt), \nabla u\in L^p(\mathbb{R}^{n+1}_+)\},$$
where $L^{p}(\mathbb{R}^{n+1}_+, t^{-p}dydt)=\{u: \int_{\mathbb{R}^{n+1}_+} t^{-p}|u|^p dydt<+\infty\}$. We have the following inequalities on the upper half space:

\begin{theorem}\label{thm1-2} Assume that $p\in (1, n+1]$ and $\beta$ satisfies \eqref{beta-p}. There is a positive sharp constant $C^*_{n+1, p, \beta}$, such that for all $u \in  \cal{D}^{1,p}_{0,0}(\mathbb{R}^{n+1}_+)$,
%{\color{red}(n=p-1, \cal{D}^{1,p}_0(\mathbb{R}^{n+1}_+)?)}$,
\begin{align}\label{p-main-2}
(\int_{\mathbb{R}^{n+1}_+} t^{-p+\frac{n+1-p}{n+1}\beta} |u|^{p+\frac{p\beta}{n+1}} dydt)^{\frac
   {n+1}{n+\beta+1}} \le C^*_{n+1, p, \beta} \int_{\mathbb{R}^{n+1}_+}|\nabla u|^p dydt.
\end{align}
Moreover, for $p=2$, the equality holds for some functions in $\cal{D}^{1,2}_{0,0}(\mathbb{R}^{n+1}_+)$  if $\beta\in (0,\frac{2(n+1)}{n-1})$ for $n\geq 2$, or $\beta>0$ for $n=1$.

In the following two cases, the extremal functions can be explicitly written out and the sharp constants can be calculated:

\noindent (1) For $\beta=1$, 
 \begin{align}\label{beta=1}
     (\int_{\mathbb{R}^{n+1}_+}t^{-\frac{n+3}{n+1}} |u|^{\frac{2n+4}{n+1}} dydt)^{\frac {n+1}{n+2}} \le  C^*_{n+1, 2, 1}\int_{\mathbb{R}^{n+1}_+}|\nabla u|^2dydt,\ \forall u\in \cal D_{0,0}^{1,2}(\mathbb{R}^{n+1}_+),
 \end{align}
where
$$
C^*_{n+1, 2, 1}=\frac{1}{2(n+1)}\big(\frac{\Gamma(n+4)}{\pi^{\frac{n}{2}}\Gamma(\frac{n}{2}+2)}\big)^{\frac{1}{n+2}},
$$
and equality in \eqref{beta=1} holds if and only if
\begin{align}\label{ex-beta=1}
u(y,t)=\frac{Ct}{(A+t)^2+|y-y_0|^2)^{\frac{n+1}2}},
\end{align}for some $A>0$, $C\in \mathbb{R}$ and $y_0\in \mathbb{R}^n.$

\noindent(2) For $\beta=2$,
\begin{equation}\label{beta=2}
 (\int_{\mathbb{R}^{n+1}_+}t^{-\frac{4}{n+1}} |u|^{\frac{2n+6}{n+1}} dydt)^{\frac {n+1}{n+3}} \le  C^*_{n+1, 2, 2}\int_{\mathbb{R}^{n+1}_+}|\nabla u|^2dydt,\ \forall u \in \cal D_{0,0}^{1,2}(\mathbb{R}^{n+1}_+)
 \end{equation}
where
\begin{align*}
C^*_{n+1, 2, 2}=\frac{1}{(n+1)(n+3)}\big(\frac{4\Gamma(n+3)}{\pi^{\frac{n+1}{2}}\Gamma(\frac{n+3}{2})}\big)^{\frac{2}{n+3}},
\end{align*}
and equality  in \eqref{beta=2} holds if and only if
\begin{align}\label{ex-beta=2}
u(y,t)=\frac{Ct}{(A^2+t^2+|y-y_0|^2)^{\frac{n+1}2}},
\end{align}
for some $A>0$,  $C\in \mathbb{R}$ and $y_0\in \mathbb{R}^n.$ 
\end{theorem}

 Note that %the Euler-Lagrange equation to \eqref{p-main-2} with $p=2$ is 
if $u\geq 0$ is an extremal function to inequality \eqref{p-main-2} for $p=2$, then up to a multiple of some constant, it holds
 	\begin{align}\label{weak-2}
    \int_{\mathbb{R}^{n+1}_+} \nabla u\nabla \phi \, dydt=\int_{\mathbb{R}^{n+1}_+} t^{-2+\frac{n-1}{n+1}\beta}u^{1+\frac{2\beta}{n+1}}\phi \, dydt, \quad \forall \phi\in \cal{D}^{1,2}_{0,0}(\mathbb{R}^{n+1}_+).
\end{align}
We define  $u\in \cal{D}^{1,2}_{0,0}(\mathbb{R}^{n+1}_+)$  to be the  weak solution to equation  
\begin{equation}\label{E-L}
	\begin{cases}
	-\Delta u=t^{-2+\frac{n-1}{n+1}\beta}u^{1+\frac{2\beta}{n+1}}, \;\;& \text{in}~ \mathbb{R}^{n+1}_+,\\
	u=0,\quad \;\;& \text{on}~ \partial \mathbb{R}^{n+1}_+,
	\end{cases}
	\end{equation}
	if equality \eqref{weak-2}
 holds. By the standard elliptic estimates and the maximum principle, we know that the nonnegative weak solutions to \eqref{E-L} are smooth and positive in the interior of the upper half space. 
	But since the boundary value of solutions to  \eqref{E-L} is zero, we can not obtain the important information on solutions
	via the method of moving sphere. To be specific, with the zero boundary condition, in the process of carrying out the method of moving sphere, we can not rule out the possibility that  moving spheres centered at certain points on the boundary never reach the critical positions while moving spheres centered at different boundary points  may reach the critical positions.
	%we have trouble to get conclusion %to say $u=0$ inside (?), as well as 
	%to use  to get the value inside.}
	%Due to the zero boundary condition, it is difficult to directly  apply the Moving Sphere Method as we use in \cite{DSWZ21} to classify the positive solutions to the related Euler-Lagrange equation of \eqref{p-main-2} with $p=2$ {\bf Maybe we add here: Note that the E-L equation to ... is..., since the boundary value is zero, we have trouble to get conclusion to say $u=0$ inside, as well as to use the method of moving sphere to get the value insider. Or add a remark later. 1-4-2022}.
	Fortunately, we are able to find the suitable transformations for the solutions, which have positive boundary values, and then we can  carry out the method of moving sphere. Meanwhile, we find the equivalence between inequality \eqref{p-main-2}  for $p=2$ and the inequality we obtained in \cite{DSWZ21} (that is, inequality \eqref{sharp-2} below). It helps us to obtain the sharp form of inequality \eqref{p-main-2}  and the classification results for $p=2$. %with $\Omega=\mathbb{R}^{n+1}_+$.
	 Actually, all the nontrivial nonnegative weak solutions %(in $\cal{D}^{1,2}_{0,0}(\mathbb{R}^{n+1}_+)$) 
	to equation \eqref{E-L} with $\beta=1$ or $\beta=2$ are of the form \eqref{ex-beta=1} or \eqref{ex-beta=2}, respectively. See Section 2.1.3  for more details. 

%\begin{remark}
 %For $1<p\leq n+1$, it is easy to check that %$$\cal{D}^{1,p}_0(\mathbb{R}^{n+1}_+)=\{u\in W^{1,1}_{loc}(\mathbb{R}^{n+1}_+): u\in L^{\frac{(n+1)p}{n+1-p}}(\mathbb{R}^{n+1}_+), \nabla u\in L^p(\mathbb{R}^{n+1}_+)\},$$
%{\color{red} it is easy to check that $$ \cal{D}^{1,p}_{0,0}(\mathbb{R}^{n+1}_+)=\{u\in W^{1,1}_{loc}(\mathbb{R}^{n+1}_+): u\in L^{p}(\mathbb{R}^{n+1}_+, t^{-p}dydt), \nabla u\in L^p(\mathbb{R}^{n+1}_+)\}.$$}
%For $p=n+1$,
%$$\cal{D}^{1,n+1}_{0,0}(\mathbb{R}^{n+1}_+)=\{u\in W^{1,1}_{loc}(\mathbb{R}^{n+1}_+): u\in L^{q}(\mathbb{R}^{n+1}_+, t^{-(n+1)}dydt), \nabla u\in L^{n+1}(\mathbb{R}^{n+1}_+)\},\quad \forall q\geq n+1.$$
%\end{remark}

%\begin{remark}
% According to Maz'ya \cite{Mazya11}, for $p=1$, $\beta\in(0,\frac{n+1}{n}]$, inequality \eqref{p-main-2} also holds.   \textbf{We may remove this remark later. 3-18-2022}
%\end{remark}

%{\bf We need to write down the corresponding inequality on $\mathbb{R}^{n+1}$ with $\delta =t.$}

\subsection{Sharp constant on bounded domains for $p=2$} Once we know the sharp constant and the explicit form  for the extremal functions of the sharp  Hardy-Sobolev inequality on the upper half space for $p=2$, we are able to study the sharp Hardy-Sobolev  inequality with  weighted distance functions on  bounded domains.
%First, we introduce some notation.

For $\beta$ satisfying
\begin{equation}\label{beta-0}
\begin{cases}
        0\leq \beta \leq \frac{2(n+1)}{n-1},& \text{ if }n\geq 2,\\
       \beta\geq 0, & \text{ if }n=1,
\end{cases}
\end{equation} and a domain $\Omega\subset\mathbb{R}^{n+1}$,% (no need to be bounded),
we define
\begin{align}\label{J}
    J_{n+1,\beta,\Omega}[u]=\frac{\int_{\Omega}|\nabla u|^2dx} { (\int_{\Omega}\delta_{\Omega}^{-2+\frac{n-1}{n+1}\beta} |u|^{2+\frac{2\beta}{n+1}} dx)^{\frac {n+1}{n+\beta+1}}},
\end{align}
and
\begin{equation}\label{inf1-1}
 \mu_{n+1, \beta}(\Omega)=\inf_{u\in  C_0^\infty (\Omega)\setminus\{0\}} J_{n+1,\beta,\Omega}[u].
 %\frac{\int_{\Omega}|\nabla u|^2dx} { (\int_{\Omega}\delta^{-2+\frac{n-1}{n+1}\beta} u^{2+\frac{2\beta}{n+1}} dx)^{\frac {n+1}{n+\beta+1}}}.
 \end{equation}
 Then $\mu_{n+1,\beta}(\Omega)>0$ if and only if inequality \eqref{p-main-1} holds for $p=2$ in $\Omega$. Apparently,  the ratio is invariant with respect to any dilation and any translation, that is,
 \begin{align}\label{dilation}
     \mu_{n+1,\beta}(\Omega_{R,x_0})=\mu_{n+1,\beta}(\Omega),\quad \forall R>0,\  x_0\in\mathbb{R}^{n+1},
 \end{align}
 where $\Omega_{R,x_0}=R\Omega+x_0=\{Rx+x_0|\, x\in\Omega\}.$

 \smallskip
 
 For simplicity, we write $\mu_{n+1,\beta}^*=\mu_{n+1,\beta}(\mathbb{R}^{n+1}_+)$. By Theorem \ref{thm1-2}, we know $$\mu_{n+1,\beta}^*=(C^*_{n+1,2,\beta})^{-1},$$ and for $\beta$ satisfying
	\begin{equation}\label{beta-strict}
	\begin{cases}
	    0<\beta<\frac{2(n+1)}{n-1}, &\text{ if }n\ge 2,\\
	    \beta>0, & \text{ if }n=1,
	    \end{cases}
	\end{equation}
$\mu_{n+1,\beta}^*$ is achieved in $\cal{D}^{1,2}_{0,0}(\mathbb{R}^{n+1}_+)$. {By contrast, the study of Hardy inequality and Sobolev inequality shows that $\mu^*_{n+1,0}=\frac{1}{4}$ ($n\geq 1$),  $\mu^*_{n+1,\frac{2(n+1)}{n-1}}=S_{n+1,2}^{-1}$ ( $n\geq 2$)}, and both constants {\it are not} achieved in $\cal{D}^{1,2}_{0,0}(\mathbb{R}^{n+1}_+)$.

%{\color{blue} Due to the zero boundary condition, it is difficult to apply the Moving Sphere Method as we use in \cite{DSWZ21}. But there is some connections between inequality \eqref{sharp1} and the inequalities in \cite{DSWZ21}, which we can use to derive Theorem \ref{thm1-3}. See more details in Section 2.3 and 3.1.} 

Naturally, we are interested in  the sharp constant of Hardy-Sobolev inequality on bounded domains. For the endpoints of the range of $\beta$, there have already been many interesting results. First, for $\beta=\frac{2(n+1)}{n-1}$ with $n\geq 2$, it is known  that for any domain $\Omega\subset\mathbb{R}^{n+1}$,
$$\mu_{n+1,\frac{2(n+1)}{n-1}}(\Omega)= \mu^*_{n+1,\frac{2(n+1)}{n-1}}=S_{n+1,2}^{-1},$$
and $\mu_{n+1,\frac{2(n+1)}{n-1}}(\Omega)$ is not achieved unless $\Omega=\mathbb{R}^{n+1}$.
Secondly, for $\beta=0$, it holds 
$$\mu_{n+1,0}(\Omega)\leq \mu^*_{n+1,0}=\frac{1}{4},$$
due to Davies \cite{Davies95} and Marcus, Mizel, and Pinchover \cite{MMP98}. Further,  if $\Omega$ is convex, then it is proved that $\mu_{n+1,0}(\Omega)= \frac{1}{4}$ (see, for example,  \cite{MS97, MMP98}). Moreover, it was showed in \cite{MMP98} that for $\Omega$ being a bounded domain with $C^2$ boundary, the sufficient and necessary condition for $\mu_{n+1,0}(\Omega)$ not to be achieved in $W^{1,2}_0(\Omega)$ is  $\mu_{n+1,0}(\Omega)=\frac 14$. 

%{For general $\beta$, it is still an open question whether $\mu_{n+1,\beta}(\Omega)$ has similar properties as $\beta=0$. We have obtained some of them. First, what you want to say???? 12-13}

 In this paper, we shall study the sharp constant of Hardy-Sobolev inequality on bounded domains  for general $\beta$. First, it is easy to show 
\begin{proposition}\label{leq}
	 Assume that $\Omega$ is a bounded  domain with Lipschitz boundary and $\beta$ satisfies \eqref{beta-0}, then  it holds
	$$0<\mu_{n+1,\beta}(\Omega)\leq \mu^*_{n+1,\beta}.$$
	\end{proposition}

Secondly, we have the following sufficient condition for $\mu_{n+1,\beta}(\Omega)$ being achieved  in $W_0^{1,2}(\Omega).$
\begin{proposition}\label{extremal-attain}
	Assume that $\Omega$ is a bounded domain with $C^1$ boundary and $\beta$ satisfies \eqref{beta-0}. If $\mu_{n+1,\beta}(\Omega)<	\mu^*_{n+1,\beta},$ then $\mu_{n+1,\beta}(\Omega)$ is achieved in $W_0^{1,2}(\Omega)$.
\end{proposition}
Proposition \ref{extremal-attain} was also obtained by Chen and Li \cite{CL07} for dimension $\geq {3}$, through a precise  description  of the related Palais-Smale sequence, while we obtain the result via  the $\varepsilon$-level sharp inequality (Lemma \ref{var-level}), which plays the important role in proving Theorem \ref{thm1-1}.

 If $\mu_{n+1,\beta}(\Omega)$ is achieved and $u\geq 0$ is the extremal function, then up to a multiple of some constant, it holds
 	\begin{align}\label{weak-3}
    \int_{\Omega} \nabla u\nabla \phi \, dydt=\int_{\Omega} t^{-2+\frac{n-1}{n+1}\beta}u^{1+\frac{2\beta}{n+1}}\phi \, dydt, \quad \forall \phi\in W^{1,2}_{0}(\Omega).
\end{align}
We define  $u\in W^{1,2}_{0}(\Omega)$  to be the  weak solution to equation  
\begin{equation}\label{E-L-1}
	\begin{cases}
	-\Delta u=\delta^{-2+\frac{n-1}{n+1}\beta}u^{1+\frac{2\beta}{n+1}}, \;\;& \text{in}~ \Omega,\\
	u=0,\quad \;\;& \text{on}~ \partial \Omega,
	\end{cases}
	\end{equation}
	if equality \eqref{weak-3}
 holds. So if  $\mu_{n+1,\beta}(\Omega)<	\mu^*_{n+1,\beta},$  there is a positive weak solution to equation \eqref{E-L-1}.
 
 \medskip

We consider some specific domains. First, for the unit ball, the following property holds (see \cite{CL07} for $n \ge 2$, which actually holds for $n=1$). %(hard to understand? 6-10-2022). 
%We state it as follows:	%\mu_{n+1,\beta}(B_R(x_0))
\begin{proposition}\label{ball-equality}
	Assume that $\beta$ satisfies \eqref{beta-0}. We have 
	\begin{align*}
		\mu_{n+1,\beta}(B_1(0))=	\mu^*_{n+1,\beta}.
	\end{align*}
	Moreover, $\mu_{n+1,\beta}(B_1(0))$ cannot be achieved by any function in $W_0^{1,2}(B_1(0))$.
\end{proposition}

As we mentioned above, for $\beta=0$ or $\frac{2(n+1)}{n-1}$,  the results are known.  In the nonlinear cases (for $\beta>0$), the approach to obtain the results for $\beta<\frac{2(n+1)}{n-1}$ is different to the case of $\beta=\frac{2(n+1)}{n-1}$ (the sharp Sobolev inequality). 
 It is  worth pointing out that  there is no positive solution to \eqref{E-L-1} with  $\Omega=B_1(0)$, $n\geq 2$ and $\beta=\frac{2(n+1)}{n-1}$% (related to the sharp Sobolev inequality)
, while there are positive radially symmetric solutions %as the critical points of $J_{n+1,\beta, B_1(0)}$ 
to \eqref{E-L-1} with  $\Omega=B_1(0)$, $n\geq 2$ and $0<\beta<\frac{2(n+1)}{n-1}$ (due to Cheng, Wei and Zhang \cite{CWZ21}). Proposition \ref{ball-equality} also implies that the minimizing sequence for $\mu_{n+1,\beta}(B_1(0))$ must blow up. 

%Although the functional $J_{n+1,\beta, B_R(x_0)}$ has no minimizer, for , there is other positive  Besides, it is still an open question whether the equality holds for convex domains. 

Secondly,  % (satisfying \eqref{beta-strict}),
we are able to find some annular domain $\Omega$, such that $0<\mu_{n+1,\beta}(\Omega)<\mu_{n+1,\beta}^*$ (see Section 4 below), thus  the sharp constant is achieved on such a domain.

We are curious about whether for any bounded non-convex domain with $C^1$ boundary, $\mu_{n+1,\beta}(\Omega)<\mu_{n+1,\beta}^*.$ If this is the case, then  $\mu_{n+1,\beta}(\Omega)$ is achieved in $W_0^{1,2}(\Omega)$.  For this purpose, it is quite interesting to study the related blow up behavior % of the minimizers of sub-critical cases as the power tend to critical case, 
and explore the role of the extremal functions of the sharp Hardy-Sobolev inequality on the upper half space in the study.  We will try to construct some auxiliary functions derived from the  extremal functions \eqref{ex-beta=1} and \eqref{ex-beta=2} to solve the question in a coming paper. 

Finally, we also speculate a positive answer to the following conjecture:

%Finally, it is hard to resist the temptation to conjecture:
%\textbf{I would like to say: first, we are curious whether for any non-convex domain, $\mu_{n+1,\beta}(\Omega)<	\mu_{n+1,\beta}^*.$ If this is the case, then  $\mu_{n+1,\beta}(\Omega)$ is achieved in $W_0^{1,2}(\Omega)$.We then speculate a positive answer to the following conjecture:

%\noindent{\bf Conjecture A.} Assume that $\Omega$ is convex and $\beta$ satisfies \eqref{beta-strict}, then $\mu_{n+1,\beta}(\Omega)=	\mu_{n+1,\beta}^*.$

\noindent{\bf Conjecture.} Assume that $\Omega$ is a bounded domain with $C^1$ boundary and $\beta$ satisfies \eqref{beta-strict}. Then   $\mu_{n+1,\beta}(\Omega)$ is achieved in $W_0^{1,2}(\Omega)$ only if  $\mu_{n+1,\beta}(\Omega)<	\mu_{n+1,\beta}^*.$

\medskip

%{\bf I will modify this paragraph after I finish reading whole paper 12-13-2021} 
The paper is organized as follows: In Section 2, we present the proof for Theorem \ref{thm1-2}, and then apply it to obtain Theorem \ref{thm1-1}. In Section 3, we consider  the  sharp constant on bounded domain for $p=2$, and give the proof of Proposition \ref{leq}, Proposition \ref{extremal-attain} and Proposition 
\ref{ball-equality}. In Section 4, we give examples for some specific domains $\Omega$,  which may have non-Lipschitz boundary point or may be unbounded.%on which, $\mu_{n+1, \beta}(\Omega)=0$ or $0<\mu_{n+1, \beta}(\Omega)<	\mu_{n+1,\beta}^*$.
\section{Hardy-Sobolev inequalities on $\mathbb{R}^{n+1}_+$ and bounded domains}

%{\bf Always remember that you are writing to others, not me.}

\subsection{Hardy-Sobolev inequality on $\mathbb{R}^{n+1}_+$}
 We first provide the  proof for Theorem \ref{thm1-2}. 
 One simple approach to obtain inequality \eqref{p-main-2} is to use the  interpolation between Hardy inequality and Sobolev inequality. Such an approach usually fails to obtain the inequality  for $p=n+1$ (that is, this allows us to get the inequality only for the case $1<p<n+1$). 
 
 Another approach for deriving inequality \eqref{p-main-2} with non-sharp constant is the classical way to derive the  Gagliardo-Nirenberg  inequality. This also yields the inequality for  the case $p=n+1$. 
 
 When it comes to the sharp constant and extremal functions for $p=2$, they can be derived from our early results in \cite{DSWZ21}. %The second method gives the whole proof including the case $p=n+1$. The third method gives the proof for $p=2$, which also helps to obtain the best constants and extremal functions for some special parameters in Section 3.
 
%5-30-2020 MJ

\subsubsection{Interpolation}

For $p>1$,  %we know a general Hardy inequality in $\mathbb{R}^{n+1}_+$ {\color{blue}(see, for example, Maz'ya \cite[Sec 1.3]{Mazya11})}:
the Hardy inequality on $\mathbb{R}^{n+1}_+$ is the  following:
\begin{align}\label{Hardy-g-p}
&\int_{\mathbb{R}^{n+1}_+} t^{-p}|u|^p dydt \leq (\frac{p}{p-1})^p \int_{\mathbb{R}^{n+1}_+} |\nabla u|^p dydt, \quad\forall\,  u\in {\cal D}_0^{1,p}(\mathbb{R}^{n+1}_+).
\end{align}
Then by H\"{o}lder inequality and Sobolev inequality, %\eqref{sob-1-1} (with $\Omega=\mathbb{R}^{n+1}_+$),
for $p\in (1,n+1)$,
\begin{align*}
&\int_{\mathbb{R}^{n+1}_+}  t^{-p+\frac{n+1-p}{n+1}\beta} u^{p+\frac{p\beta}{n+1}} dydt\\
\leq &\big(\int_{\mathbb{R}^{n+1}_+} t^{-p}|u|^pdydt \big)^{1-\frac{(n+1-p)\beta}{p(n+1)}}(\int_{\mathbb{R}^{n+1}_+} |u|^{\frac{p(n+1)}{n+1-p}}dx)^{\frac{(n+1-p)\beta}{p(n+1)}}\\
\leq& C\big(\int_{\mathbb{R}^{n+1}_+} |\nabla u|^p dydt\big)^{\frac{n+
\beta+1}{n+1}}.
\end{align*}
This gives the proof of inequality \eqref{p-main-2} in Theorem \ref{thm1-2} for the case $1<p<n+1$.

\bigskip

\subsubsection{ Gagliado-Nirenberg type inequalities in $\mathbb{R}^{n+1}_+$}

We only consider $u\in C^\infty_0({\mathbb{R}^{n+1}_+}) $ since the general case can be gotten by approximation.  So we assume $u(y,0)=0$. With this extra condition we can extend our inequalities obtained  in the early paper \cite{DSWZ21}. First, similar to Lemma 2.1 and Lemma 2.2 in \cite{DSWZ21}, we have the following two lemmas.

\begin{lemma}\label{lem2-1}  For any  $k\ne0$ and $u\in C^\infty_0({\mathbb{R}^{n+1}_+})$, 
\begin{align}\label{eq:lem2-1}
     \int_{\RpN}t^{k-1}|u|dydt\leq \frac{1}{|k|}\int_{\RpN}t^{k}|\nabla u|dydt.
\end{align}
\end{lemma}

%\begin{proof}
%Observe that for $k\ne 0$, $u\in C^\infty_0({\mathbb{R}^{n+1}_+})$,
%$$
%\int_0^\infty t^{k-1} u(y, t)dt= -\frac 1{k}\int_{0}^\infty  \frac {\partial u (y, t) }{\partial t}  \cdot t^{k} dt.
%$$
%Integrating with respect to $y$ on both sides gives the desired inequality.
%\end{proof}

 \begin{lemma}\label{lem2-2}   For any  $k$ and $u\in C^\infty_0({\mathbb{R}^{n+1}_+})$, 
\begin{align}\label{lem2-2-eq}
   \big(\int_{\RpN}t^{\frac{n+1}{n}k}|u|^{\frac{n+1}{n}}dydt\big)^{\frac{n}{n+1}}\leq  2^{\frac{1}{n}}\int_{\RpN} t^{k}|\nabla u|dydt.
\end{align}
\end{lemma}

\begin{remark}\label{rem2-1}
	In Lemma \ref{lem2-1}, for $k<0$, we need to assume that $u(y, 0)=0$. For $k>0$, this condition is not needed --- this was proved in  \cite[Lemma 2.1]{DSWZ21}.
	%{\color{blue}The two constants in Lemma \ref{lem2-1} and Lemma \ref{lem2-2} are not optimal.}
\end{remark}

\smallskip

By Lemma \ref{lem2-1}, Lemma \ref{lem2-2} and H\"{o}lder inequality, we easily obtain the following Gagliado-Nirenberg type inequalities.
\begin{proposition}\label{prop2-1}
	(1) Assume that $k \ne -n$ or $0$, and $l$ satisfies
	\begin{equation}\label{l-condition}
	\begin{cases}
	l\in [k-1, \frac {n+1}n k], & \text{ if } k \in(-n, 0) \cup (0, \infty),\\
	l\in [ \frac {n+1}n k, k-1], & \text{ if } k<-n.
	\end{cases}
	\end{equation}
	 %$k \in(-n, 0) \cap (0, \infty)$, $  l\in [k-1, \frac {n+1}n k]$ or $k<-n, \ l\in [ \frac {n+1}n k, k-1]$,
   Then	there is a positive constant $C=C(n,k)$ such that for all $u\in C_0^\infty ({\mathbb{R}^{n+1}_+})$,
	\begin{equation}\label{GGN-1}
	(\int_{\mathbb{R}^{n+1}_+ }t^l |u|^{\frac{n+l+1}{n+k}} dydt)^{\frac{n+k}{n+l+1}} \le C \int_{\mathbb{R}^{n+1}_+ }t^k|\nabla u| dy dt.
		\end{equation}
	(2) For $k=-n$ and $ l=-n-1$, let $q\in [1,\frac{n+1}{n}]$. Then there is a positive constant $C=C(n)$ such that for all $u\in C_0^\infty ({\mathbb{R}^{n+1}_+})$,
	\begin{equation}\label{GGN-1--n}
	(\int_{\mathbb{R}^{n+1}_+ }t^{-n-1}|u|^{q} dydt)^{\frac{1}{q}} \le C \int_{\mathbb{R}^{n+1}_+ }t^{-n} |\nabla u| dy dt.
	\end{equation}
		\end{proposition}

\begin{remark}
According to Maz'ya \cite[(2.1.34)]{Mazya11}, for $k=0$ and $l\in (-1,0]$, inequality \eqref{GGN-1} is also true, even for $u(y,0)\neq 0$. But our approach does not work for this case. %(see Maz'ya \cite[(2.1.34)]{Mazya11}). %{\color{red} (But we do not know how to deal with this case in our method.Jingbo and his students are working on this case 12-20-2021)}
\end{remark}

As an application of Proposition \ref{prop2-1}, we have the following corollary,  which yields  inequality \eqref{p-main-2} in Theorem \ref{thm1-2} (by choosing $pk-(p-1)l=0 $ in \eqref{GGN-o-1-p} and $p=n+1$ in \eqref{GGN-o-1-n+1}).
 
\begin{corollary}\label{cor2-6-p}
	(1) Let $k\neq -n$ or $0,$ $l$ satisfy \eqref{l-condition}, $p\in[1,n+1],$ or $k=0,\ l\in (-1,0],\ p\in(1,n+1]$.
	%{\color{red}( Mazya: it also holds for $k=0,\ l\in (-1,0],\ p=1$. Can we have another proof?)}.
	%\begin{equation}\label{l-condition-1}
%	\begin{cases}
%	l\in [k-1, \frac {n+1}n k], & \text{ if } k >-n,\\
%	l\in [ \frac {n+1}n k, k-1], & \text{ if } k<-n.
%	\end{cases}
%	\end{equation}
	 We further assume that
	 \begin{align}\label{l:p=n+1}
	 l\neq \frac {n+1}n k, \ \text{ if }  p=n+1.
	 \end{align} Then
	there is a positive constant $C>0$ such that, for all $u\in C_0^\infty ({\mathbb{R}^{n+1}_+})$,
	%$n+2k\ge l+1$ and $l \le \frac{n+1}n k$,
	\begin{equation}\label{GGN-o-1-p}
	\big(\int_{\mathbb{R}^{n+1}_+}t^l|u|^{\frac{p(n+l+1)}{pk-(p-1)l+n+1-p}}dydt\big)^{\frac{pk-(p-1)l+n+1-p}{n+l+1}} \leq C\int_{\mathbb{R}^{n+1}_+}t^{pk-(p-1)l}|\nabla u|^p dydt.
	\end{equation}
	(2) For $k=-n$, $l=-n-1$, let $p\in [1, n+1]$,
	and  $s$ satisfy
	\begin{equation}\label{s-condition}
	\begin{cases}
	s\in [p, \frac{(n+1)p}{n+1-p}],& \text{ if } p<n+1,\\
	s\ge p, & \text{ if } p=n+1.
	\end{cases}
	\end{equation} Then there is a positive constant $C>0$, such that for all $u\in C_0^\infty ({\mathbb{R}^{n+1}_+})$,
	\begin{align}\label{GGN-o-1-n+1}
	\big(\int_{\mathbb{R}^{n+1}_+}t^{-n-1}|u|^{s}dydt\big)^{\frac{p}{s}} \leq C\int_{\mathbb{R}^{n+1}_+}t^{p-n-1}|\nabla u|^p dydt.
	\end{align}
	
\end{corollary}

\begin{proof}
(1) For the case $k\neq -n$ or $0 $,
 by \eqref{l-condition} and \eqref{l:p=n+1}, it is easy to check that $pk-(p-1)l +n+1-p\ne 0$. Applying Proposition \ref{prop2-1} (1) to $u^{\frac{p(n+k)}{pk-(p-1)l+n+1-p}}$, and by H\"older inequality, we have
\begin{equation*}
\begin{split}
&\big(\int_{\mathbb{R}^{n+1}_+}t^l|u|^{\frac{p(n+l+1)}{pk-(p-1)l+n+1-p}}dydt\big)^{\frac{n+k}{n+l+1}}\\
\leq&C\int_{\mathbb{R}^{n+1}_+}t^k |u|^{\frac{(p-1)(n+l+1)}{pk-(p-1)l+n+1-p}}|\nabla u|dydt\\		\leq&C\big(\int_{\mathbb{R}^{n+1}_+}t^l|u|^{\frac{p(n+l+1)}{pk-(p-1)l+n+1-p}}dydt\big)^{\frac{p-1}{p}}\big(\int_{\mathbb{R}^{n+1}_+}t^{pk-(p-1)l}|\nabla u|^p dydt\big)^{\frac{1}{p}},
\end{split}
\end{equation*}
which gives \eqref{GGN-o-1-p}.
%\begin{equation}\label{+1-p}
%\big(\int_{\mathbb{R}^{n+1}_+}t^l|u|^{\frac{p(n+l+1)}{pk-(p-1)l+n+1-p}}dydt\big)^{\frac{pk-(p-1)l+n+1-p}{n+l+1}}\leq C\int_{\mathbb{R}^{n+1}_+}t^{pk-(p-1)l}|\nabla u|^p dydt.
%\end{equation}

Next, we consider the case $k=0$. Here $ l\in (-1,0],\ p\in(1,n+1]$, and additionally, $l\neq 0$ if $p=n+1$. We need to show that
\begin{align}\label{GGN:k=0}
		\big(\int_{\mathbb{R}^{n+1}_+}t^l|u|^{\frac{p(n+l+1)}{-(p-1)l+n+1-p}}dydt\big)^{\frac{-(p-1)l+n+1-p}{n+l+1}} \leq C\int_{\mathbb{R}^{n+1}_+}t^{-(p-1)l}|\nabla u|^p dydt.
\end{align}
Take $\tilde{k}\in [\frac{n}{n+1}l, l+1]\backslash\{0\}$. Applying Proposition \ref{prop2-1} (1) to $u^{\frac{p(n+\tilde{k})}{-(p-1)l+n+1-p}}$ with $k$ replaced by $\tilde{k}$, and by H\"older inequality, we have
\begin{equation}\label{GGN:k=0-1}
\begin{split}
&\big(\int_{\mathbb{R}^{n+1}_+}t^l|u|^{\frac{p(n+l+1)}{-(p-1)l+n+1-p}}dydt\big)^{\frac{n+\tilde{k}}{n+l+1}}\\
\leq&C\int_{\mathbb{R}^{n+1}_+}t^{\tilde{k}} |u|^{\frac{p(n+\tilde{k})}{-(p-1)l+n+1-p}-1}|\nabla u|dydt\\		\leq&C\big(\int_{\mathbb{R}^{n+1}_+}t^{-(p-1)l}|\nabla u|^p dydt\big)^{\frac{1}{p}}\big(\int_{\mathbb{R}^{n+1}_+}t^{\frac{p\tilde{k}+(p-1)l}{p-1}}|u|^{\frac{p(n+\frac{p\tilde{k}+(p-1)l}{p-1}+1)}{-(p-1)l+n+1-p}}dydt\big)^{\frac{p-1}{p}}.
\end{split}
\end{equation}
For the second integral in the last line of \eqref{GGN:k=0-1}, we apply \eqref{GGN-o-1-p} with $k,\ l$ replaced by $\tilde{k}, \frac{p\tilde{k}+(p-1)l}{p-1}$ respectively, then
\begin{align*}
	\big(\int_{\mathbb{R}^{n+1}_+}t^{\frac{p\tilde{k}+(p-1)l}{p-1}}|u|^{\frac{p(n+\frac{p\tilde{k}+(p-1)l}{p-1}+1)}{-(p-1)l+n+1-p}}dydt\big)^{\frac{-(p-1)l+n+1-p}{n+\frac{p\tilde{k}+(p-1)l}{p-1}+1}}\leq C\int_{\mathbb{R}^{n+1}_+}t^{-(p-1)l}|\nabla u|^p dydt,
\end{align*}
where we further choose $\tilde{k}$ to satisfy $\tilde{k}\in [-(p-1)(l+1), -\frac{n(p-1)}{n+1-p}l] $ if $1<p<n+1$ or $\tilde{k}\geq -(p-1)(l+1)$ if $p=n+1$, in order that conditions \eqref{l-condition} and \eqref{l:p=n+1} hold with $k,\ l$ replaced by $\tilde{k}, \frac{p\tilde{k}+(p-1)l}{p-1}$ respectively.
Taking above inequality back to \eqref{GGN:k=0-1}, we obtain \eqref{GGN:k=0}.

(2) Assume that $q$ satisfies $q\in [1, \frac{n+1}{n}]$ if $p<n+1$ or $q\in [1,\frac{n+1}{n})$ if $p=n+1$. It is easy to check that $p+q-pq\neq 0$. Applying Proposition \ref{prop2-1} (2) to $u^{\frac{p}{p+q-pq}}$,
by H\"older inequality, we have
\begin{equation*}
\begin{split}
&\big(\int_{\mathbb{R}^{n+1}_+}t^{-n-1}|u|^{\frac{pq}{p+q-pq}}dydt\big)^{\frac{1}{q}}\\
\leq&C\int_{\mathbb{R}^{n+1}_+}t^{-n} |u|^{\frac{(p-1)q}{p+q-pq}}|\nabla u|dydt\\		\leq&C\big(\int_{\mathbb{R}^{n+1}_+}t^{-n-1}|u|^{\frac{pq}{p+q-pq}}dydt\big)^{\frac{p-1}{p}}\big(\int_{\mathbb{R}^{n+1}_+}t^{p-n-1}|\nabla u|^p dydt\big)^{\frac{1}{p}}.
\end{split}
\end{equation*}
Then we can get \eqref{GGN-o-1-n+1} by taking $s=\frac{pq}{p+q-pq}$.
%\begin{align*}
%	\big(\int_{\mathbb{R}^{n+1}_+}t^{-n-1}|u|^{s}dydt\big)^{\frac{p}{s}}\leq C\int_{\mathbb{R}^{n+1}_+}t^{p-n-1}|\nabla u|^p dydt.
%\end{align*}
%Furthermore, we can obtain Theorem \ref{thm1-2} (2) when taking $p=n+1$ and $\tau=s-n-1$.
\end{proof}

\bigskip
We now give the proof for inequality \eqref{p-main-2} in Theorem \ref{thm1-2}.

%\noindent{\bf Proof of Theorem \ref{thm1-2}.}
	(1) For any $p\in (1,n+1)$, we can choose $k,l$ in \eqref{GGN-o-1-p} such that $pk-(p-1)l=0$, then for $l\in[-p,0]$ and any $ u \in C_0^{\infty}(\mathbb{R}^{n+1}_+)$,
		 \begin{align}\label{pk-(p-1)l=0}
	\big(\int_{\mathbb{R}^{n+1}_+}t^l|u|^{\frac{p(n+l+1)}{n+1-p}}dydt\big)^{\frac{n+1-p}{n+l+1}} \leq C(n,p,l)\int_{\mathbb{R}^{n+1}_+}|\nabla u|^p dydt.
	\end{align}
 Taking $l=-p+\frac{n+1-p}{n+1}\beta$ in \eqref{pk-(p-1)l=0}, we have $\frac{p(n+l+1)}{n+1-p}=p+\frac{p\beta}{n+1}$. Therefore, for $\beta\in[0,\frac{p(n+1)}{n+1-p}]$, % if $p\in (1,n+1)$, or $\tau\in (0,1]$ if $p=1$,
 \eqref{pk-(p-1)l=0} becomes
$$
(\int_{\mathbb{R}^{n+1}_+} t^{-p+\frac{n+1-p}{n+1}\beta} u^{p+\frac{p\beta}{n+1}} dydt)^{\frac {n+1}{n+\beta+1}} \le C \int_{\mathbb{R}^{n+1}_+}|\nabla u|^p dydt,\ \forall u\in C_0^{\infty}(\mathbb{R}^{n+1}_+).
$$
Then by the approximation argument, we know that the inequality also holds for $u\in \cal{D}^{1,p}_{0,0}(\mathbb{R}^{n+1}_+)$.

(2) The case $p=n+1$ follows from Corollary \ref{cor2-6-p} (2) with $p=n+1,$ and $\beta=s-n-1$.
%\hfill$\Box$

% \Big(\begin{remark}\label{rem-k=0}{\color{blue} According to Mazya, in Corollary \ref{cor2-6-p} Part (1), it also holds for the case $k=0, \ l\in (-1,0),\ p=1.$	Whether can we send $p\to 1^+$ in \eqref{pk-(p-1)l=0}? 	For now, the answer is no. We have  {\color{red}$C(n,p,l)\geq \frac{C}{(p-1)^\gamma}$} for some $\gamma>0$. However, this $C(n,p,l)$ in \eqref{pk-(p-1)l=0} is not the best constant.}
%	\end{remark}\Big)

%\begin{corollary}\label{cor2-7}
%Assume $n\geq 2$, $k \in(-n,  \infty)$ and $ k-1\le 2k\le \frac {n+1}n k$.  Then,
% there is a positive constant $C>0$ such that, for all $u\in C_0^\infty ({\mathbb{R}^{n+1}_+})$,
%\begin{equation}\label{GGN-o-2}
%\big(\int_{\mathbb{R}^{n+1}_+}t^{2k}|u|^{\frac{2(n+1+2k)}{n-1}}dydt\big)^{\frac{n-1}{n+2k+1}} \leq C\int_{\mathbb{R}^{n+1}_+}|\nabla u|^2 dydt.
%\end{equation}
%\end{corollary}

%\noindent{\bf Proof of Corollary \ref{cor2-7}}. For $k\ne 0$, take $l=2k$ in Corollary %\ref{cor2-6} we obtain the inequality. For $k=0$, it is the usual Sobolev inequality.

%\begin{remark}
%If $k\le 0$, $k-1 \le 2k\le \frac {n+1}n k$ implies: $-1 \le k$ and $ n\ge 1$. If $k> 0$, $k-1 \le 2k\le \frac {n+1}n k$ implies:  $ n\le 1$. Nothing holds for $k>0$. In conclusion: Corollary \ref{cor2-2} only holds for $n \ge 2,$ and $-1 \le k \le 0.$ (\textbf{More over, this is covered by following Theorem \ref{maintheorem-1}})
%\end{remark}

\bigskip

\subsubsection{Sharp constant and  extremal functions on $\mathbb{R}^{n+1}_+$}

%{\bf This subsection is not well written. Again, it looks like you just want to convince ME how you can get the results. I do not feel that others want to read through. To observe the  equivalent relation is NOT trivial, to verify it is not so trivial. We need to emphasize our new inequality. And the classification can NOT be obtained by other way at this moment. 1-6-2022  }

 %Recall that 
%\begin{align*}
%    \mu_{n+1,\beta}^*=\inf_{u\in  \cal{D}^{1,2}_{0,0}(\mathbb{R}^{n+1}_+)\setminus\{0\}}
 %   \frac{\int_{\mathbb{R}^{n+1}_+}|\nabla u|^2dx} { (\int_{\mathbb{R}^{n+1}_+}t^{-2+\frac{n-1}{n+1}\beta} |u|^{2+\frac{2\beta}{n+1}} dx)^{\frac {n+1}{n+\beta+1}}}.
%\end{align*}

 By the concentration compactness principle (see, for example,  Lions \cite{Lions1985a,Lions1985b}), it is not hard to obtain the existence of extremal functions for $\mu_{n+1,\beta}^*$ if $\beta$ satisfies \eqref{beta-strict}. %$\beta\in (0, \frac{2(n+1)}{n-1})$ for $n\geq 2$ or $\beta>0$ for $n=1$. 
 See a similar argument given in \cite{DSWZ21}. In this section, we prove the rest part of Theorem \ref{thm1-2}. That is, we classify all  positive weak solution to equation \eqref{E-L}, which are also the extremal functions of $\mu_{n+1,\beta}^*$ due to the uniqueness of the solution. In two cases $\beta=1$ and $\beta=2$, the extremal functions can be written out explicitly, and then the corresponding sharp constants can be computed.
 
One of the common approach to obtain the classification result nowadays is to use the method of moving sphere. If the moving spheres can reach the critical positions before they all shrink to one point, then we can obtain essential information on solutions via the key Li and Zhu's calculus lemmas (see \cite{LZ95}, or for example, Lemma 5.2 in  our previous work \cite{DSWZ21}).
%get the spherical inverse symmetry{\color{red}(?)} of nonnegative solution $u$ to equation \eqref{E-L}, and after suitable Kelvin transformation, $u$ changes to a radially symmetric function on a ball. Then the equation is reduced to an ODE and the final step is to solve the ODE.
Unfortunately, due to the  zero boundary condition in equation \eqref{E-L}, we can not obtain any useful information about the solutions via using the method of moving sphere directly.   To overcome this difficulty, we  consider the new function $\frac{u}{t}$ and  prove that it is  positive on the boundary. Then the method of moving sphere can be applied to $\frac{u}{t}$ and the related equation.

First, we denote
$$C^\infty_0(\overline{\mathbb{R}^{n+1}_+})=\{v|_{\overline{\mathbb{R}^{n+1}_+}}:\ v\in C^\infty_0(\mathbb{R}^{n+1})\},$$
and for $\alpha>0$, 
 $\cal{D}^{1,2}_{\alpha}(\mathbb{R}^{n+1}_+)$ to be the completion of $C^\infty_0(\overline{\mathbb{R}^{n+1}_+})$ under the norm
$$\|v\|_{\cal{D}^{1,2}_{\alpha}(\mathbb{R}^{n+1}_+)}=\big(\int_{\mathbb{R}^{n+1}_+}t^{\alpha} |\nabla v|^2 dydt\big)^{\frac{1}{2}}.$$
It follows from Lemma 7.2 in \cite{DSWZ21} that for $\alpha\geq 1$, $C^\infty_0(\mathbb{R}^{n+1}_+)$ is dense in $\cal{D}^{1,2}_{\alpha}(\mathbb{R}^{n+1}_+)$.
For $u$ and $v=\frac{u}{t},$ we have the following property, for which we will provide the proof later.

\begin{proposition}\label{u-v-1}$u\in \cal{D}^{1,2}_{0,0}(\mathbb{R}^{n+1}_+)$   if and only if  $v=t^{-1}u\in {\cal D}_{2}^{1,2}(\mathbb{R}^{n+1}_+)$. Moreover, for $u\in\cal{D}_{0,0}^{1,2}(\mathbb{R}^{n+1}_+)$ and $\beta$ satisfying  \eqref{beta-0}, it holds 
\begin{align}\label{equal-1}
\int_{\mathbb{R}^{n+1}_+}|\nabla u|^2 dydt=\int_{\mathbb{R}^{n+1}_+}t^2|\nabla v|^2 dydt,
\end{align}
and \begin{align}\label{equal-2}
\int_{\mathbb{R}^{n+1}_+}t^{-2+\frac{n-1}{n+1}\beta} |u|^{2+\frac{2\beta}{n+1}}dydt=\int_{\mathbb{R}^{n+1}_+}t^\beta |v|^{2+\frac{2\beta}{n+1}}dydt.
\end{align}
\end{proposition} 

\medskip

Due to Proposition \ref{u-v-1}, we know that: for $\beta$ satisfying \eqref{beta-0}, 
\begin{align}\label{sharp-1}
    (\int_{\mathbb{R}^{n+1}_+}t^{-2+\frac{n-1}{n+1}\beta} |u|^{2+\frac{2\beta}{n+1}} dydt)^{\frac {n+1}{n+\beta+1}}\leq (\mu_{n+1,\beta}^*)^{-1} \int_{\mathbb{R}^{n+1}_+}|\nabla u|^2dydt,\ \forall u\in \cal{D}^{1,2}_{0,0}(\mathbb{R}^{n+1}_+)
\end{align}
is equivalent to the following sharp inequalities:
\begin{equation}\label{sharp-2}
(\int_{\mathbb{R}^{n+1}_+ }t^\beta |v|^{2+\frac{2\beta}{n+1}}dydt)^{\frac{n+1}{n+\beta+1}} \le (\mu_{n+1,\beta}^*)^{-1}\int_{\mathbb{R}^{n+1}_+ }t^{2}|\nabla v|^2 dy dt,\ \forall v\in \cal{D}_2^{1,2}(\mathbb{R}^{n+1}_+).
\end{equation} 
Further, the extremal functions of inequality \eqref{sharp-2} satisfy  	\begin{equation}\label{weak-1}
\int_{\mathbb{R}^{n+1}_+ }t^{2} \nabla v \cdot \nabla \phi dy dt =\int_{\mathbb{R}^{n+1}_+ }t^\beta v^{1+\frac{2\beta}{n+1}}\phi dydt,\ \forall  \phi \in {\cal D}_{2}^{1,2}(\mathbb{R}^{n+1}_+).
\end{equation}
We define $v\in \cal{D}^{1,2}_{2}(\mathbb{R}^{n+1}_+)$ to be the weak solution to \begin{equation}\label{genequ-1}
	\begin{cases}
	-div(t^2 \nabla v)=t^{\beta}v^{1+\frac{2\beta}{n+1}},\quad \;\;& \text{in}~ \mathbb{R}^{n+1}_+,\\
	\lim_{t \to 0^+} t^2\frac{\partial v}{\partial t}=0,\quad \;\;& \text{on}~ \partial \mathbb{R}^{n+1}_+,
	\end{cases}
	\end{equation} if \eqref{weak-1} holds. Then we have 
\begin{proposition}\label{so-equiv}
Assume that $\beta$ satisfies \eqref{beta-0}. Then $u\in \cal{D}^{1,2}_{0,0}(\mathbb{R}^{n+1}_+)$ is a  weak solution to \eqref{E-L} if and only if  $v=t^{-1}u\in {\cal D}_{2}^{1,2}(\mathbb{R}^{n+1}_+)$  is the weak solution to \eqref{genequ-1}.

\end{proposition}

It follows that the study of inequality \eqref{sharp-1} is reduced to the study of \eqref{sharp-2}, and the classification of nonnegative weak solutions to \eqref{E-L}
is reduced to the classification of nonnegative weak solutions to \eqref{genequ-1}.
It turns out, inequality \eqref{sharp-2} is a special case (for $\alpha=2$) of a more general sharp weighted Sobolev inequality obtained  in our former work \cite{DSWZ21} :
\begin{equation}\label{GGN-2-1}
(\int_{\mathbb{R}^{n+1}_+ }t^\beta |v|^{\frac{2(n+\beta+1)}{n+\alpha-1}}dydt)^{\frac{n+\alpha-1}{n+\beta+1}} \le  S_{n+1, \alpha,\beta}^{-1} \int_{\mathbb{R}^{n+1}_+ }t^{\alpha}|\nabla v|^2 dy dt,\ \forall v\in {\cal D}_{\alpha}^{1,2}(\mathbb{R}^{n+1}_+),
\end{equation}
where $\alpha>0,\beta>-1,\frac{n-1}{n+1}\beta\leq \alpha\leq\beta+2$. It is easy to see that $\mu_{n+1,\beta}^*=S_{n+1,2,\beta}.$
 We have shown that a nontrivial nonnegative weak solution  $v$ to \eqref{genequ-1} must be positive in $\overline{\mathbb{R}^{n+1}_+}$.
 %Then the method of moving sphere can be applied to get the spherical inverse symmetry{\color{red}(?)} of $v$, that is, for any $b\in\partial \mathbb{R}^{n+1}_+$, there is $\lambda_b>0$ such that 
%\begin{align}\label{v-sym}
%v(y,t)=\Big(\frac{\lambda_b}{|(y,t)-b|}\Big)^{n+1}v\Big(b+\frac{\lambda_b^2\big((y,t)-b\big)}{|(y,t)-b|^2}\Big),\quad \forall (y,t)\in\overline{\mathbb{R}^{n+1}_+}\backslash\{b\}.
%\end{align}
%Then by suitable Kelvin transformation %(with respect to $\partial B_1(-Ae_{n+1})$), 
%$v$ is reduced to a radially symmetric function %( on $B_{\frac{1}{2A}}(-\frac{2A^2-1}{2A}e_{n+1}%)$) 
%and equation \eqref{genequ-1} is reduced to an ODE (i.e. \eqref{ode-0} below). By solving the ODE, we get the  classification result, which is Theorem 1.6 in \cite{DSWZ21} with $\alpha=2$ and stated as follows:
Moreover, we have the following theorem.
\begin{theorem}\label{DSWZ-1}(\cite[Theorem 1.6]{DSWZ21})
Assume that $\beta$ satisfy \eqref{beta-strict}.
Let $v\in {\cal D}_{2}^{1,2}(\mathbb{R}^{n+1}_+)$ be a positive weak solution to equation \eqref{genequ-1}. We have, up to the multiple of some constants,
\begin{equation}\label{type-0}
v(y, t)=\big(\frac{1}{|y-y^o|^2+(t+A)^2}\big)^{\frac{n+1}2} \psi\Big(\big|\frac{(y-y^o, t+A)}{|y-y^o|^2+(t+A)^2}-(0, \frac1{2A})\big|\Big),
\end{equation}
where $y^o\in\mathbb{R}^n$, $ A>0$, $\psi(r)>0$  and $\psi\in C^2[0,\frac1{2A})\cap C^0[0,\frac1{2A}]$ satisfies an ODE:
 \begin{equation}\label{ode-0}
	\begin{cases}
		\psi''+(\frac{n}{r}-\frac{4  r}{\frac{1}{4A^2}-r^2})\psi'-\frac{2(n+1)}{\frac{1}{4A^2}-r^2}\psi
		=-KA^{\beta-2}(\frac{1}{4A^2}-r^2)^{\beta-2} \psi^{1+\frac{2\beta}{n+1}}, \; 0<r <\frac{1}{2A},\\
		\psi(\frac{1}{2A})=A^{\frac{n+1}{2}},\ \psi'(0)=0,\   \lim_{r\to (\frac{1}{2A})^{-}}\big(\frac{1}{4A^2}-r^2\big)^2 \psi'(r)=0,
	\end{cases}
\end{equation}
for one constant $K>0$ independent of  $A$. 
 Furthermore, there is only one positive solution to ODE \eqref{ode-0}.

Moreover, in the following two cases, the solutions can be  explicitly written out.

\noindent  1).  For $\beta=1$,  up to the multiple of some constant, $v(y, t)$ must be in the form of
\begin{equation}\label{sol-1-0}
v(y,t)= \Big(\frac{A}{(A+t)^2+|y-y^o|^2}\Big)^{\frac{n+1}{2}},
\end{equation}
where $ A>0$, $y^o \in \mathbb{R}^{n}$,
and
$$%\mu_{n+1, 1}^*
S_{n+1,2,1}=2(n+1)
\big[\pi^\frac n2\frac{\Gamma(\frac{n}2+2)}{\Gamma(n+4)}\big]^{\frac{1}{n+2}}.
$$

\noindent  2).   For $\beta=2$, up to the multiple of some constant, $v(y, t)$ must be in the form of
\begin{equation}\label{sol-2-0}
v(y,t)= \Big(\frac{A}{A^2+t^2+|y-y^o|^2}\Big)^{\frac{n+1}{2}},
\end{equation}
where $A>0$, $y^o \in \mathbb{R}^{n}$, and
$$%\mu_{n+1, 2}^*
S_{n+1,2,2}=(n+1)(n+3)
\big[\frac{\pi^{\frac {n+1}{2}}}{4}\frac{\Gamma(\frac{n+3}2)}{\Gamma({n+3})}
\big]^{\frac{2}{n+3}}.
$$
\end{theorem}

\medskip

With the help of Proposition \ref{so-equiv}
and Theorem \ref{DSWZ-1}, we hereby complete the proof of Theorem \ref{thm1-2}. So, we are only left to prove Proposition \ref{u-v-1} and Proposition \ref{so-equiv}.

\medskip

\noindent{\bf Proof of Proposition \ref{u-v-1}.}
 Let $u\in C_0^{\infty}(\mathbb{R}^{n+1}_+)$, then $v=\frac{u}{t}$ is also in $C_0^{\infty}(\mathbb{R}^{n+1}_+)$. Direct  calculation yields
\begin{align*}
\int_{\mathbb{R}^{n+1}_+}|\nabla u|^2 dydt&=\int_{\mathbb{R}^{n+1}_+}t^2|\nabla v|^2 dydt+2\int_{\mathbb{R}^{n+1}_+} tv\partial_t vdydt+ \int_{\mathbb{R}^{n+1}_+}v^2 dydt\\
&=\int_{\mathbb{R}^{n+1}_+}t^2|\nabla v|^2 dydt.
\end{align*}
%Using integration by parts, it holds 
%$$2\int_{\mathbb{R}^{n+1}_+} tv\partial_t %vdydt=- \int_{\mathbb{R}^{n+1}_+}v^2 dydt.$$
%Then 
%\begin{align*}
%\int_{\mathbb{R}^{n+1}_+}|\nabla u|^2 %dydt=\int_{\mathbb{R}^{n+1}_+}t^2|\nabla v|^2 %dydt.
%\end{align*}
Also, it is easy to check that
$$\int_{\mathbb{R}^{n+1}_+}t^{-2+\frac{n-1}{n+1}\beta} |u|^{2+\frac{2\beta}{n+1}}dydt=\int_{\mathbb{R}^{n+1}_+}t^\beta |v|^{\frac{2(n+\beta+1)}{n+1}}dydt.$$
That is, \eqref{equal-1} and \eqref{equal-2} hold for $u\in C_0^\infty(\mathbb{R}^{n+1}_+)$.

Let $v\in \cal{D}_{2}^{1,2} ({\mathbb{R}^{n+1}_+})$. Since $C^\infty_0(\mathbb{R}^{n+1}_+)$ is dense in $\cal{D}^{1,2}_{2}(\mathbb{R}^{n+1}_+)$, there is $\{v_j\}\subset  C_0^\infty ({\mathbb{R}^{n+1}_+})$, such that 
\begin{align*}
    \int_{\mathbb{R}^{n+1}_+}t^2 |\nabla v_j-\nabla v|^2 dydt\to 0,\ \text{ as }j\to \infty.
\end{align*}
Then as $i,j\to \infty$,
\begin{align*}
    \big(\int_{\mathbb{R}^{n+1}_+ }t^{\beta} |v_i-v_j|^{\frac{2(n+\beta+1)}{n+1}}dydt\big)^{\frac{n+1}{n+\beta+1}}\leq S_{n+1,2,\beta}^{-1}\int_{\mathbb{R}^{n+1}_+}t^2 |\nabla v_i-\nabla v_j|^2 dydt\to 0,\ %\text{ as }i , j\to \infty,
\end{align*}
which implies that 
\begin{align}\label{a.e.-1}
v_j\to v \text{ a.e. in } \mathbb{R}^{n+1}_+.
\end{align} Set 
\begin{align}\label{u_j-1}
u_j=tv_j,\quad u=tv,
\end{align}
then $u_j\in  C_0^\infty ({\mathbb{R}^{n+1}_+})$, and we need to prove $u\in \cal{D}_{0,0}^{1,2} ({\mathbb{R}^{n+1}_+})$. Since
\begin{align*}
    \int_{\mathbb{R}^{n+1}_+} |\nabla u_i-\nabla u_j|^2dydt=\int_{\mathbb{R}^{n+1}_+ }t^2|\nabla v_i-\nabla v_j|^2dy dt\to 0, \text{ as }i,j\to \infty,
\end{align*}
 there is $\tilde{u}\in \cal{D}_{0,0}^{1,2} ({\mathbb{R}^{n+1}_+})$, such that $u_j\to \tilde{u}$ in $ \cal{D}_{0,0}^{1,2} ({\mathbb{R}^{n+1}_+})$.  Similar to \eqref{a.e.-1}, we have 
$$u_j\to \tilde{u} \text{ a.e. in } \mathbb{R}^{n+1}_+.$$
Then by \eqref{a.e.-1} and \eqref{u_j-1}, we conclude that $u=\tilde{u}$ and $u\in \cal{D}_{0,0}^{1,2} ({\mathbb{R}^{n+1}_+})$.
This yields the sufficiency to the conclusion.
Similarly, we can obtain the necessity,
and  show that \eqref{equal-1} and \eqref{equal-2} hold for $u\in \cal{D}_{0,0}^{1,2} ({\mathbb{R}^{n+1}_+})$.
\hfill$\Box$

\bigskip

\noindent{\bf Proof of Proposition \ref{so-equiv}.}
For $u\in  \cal{D}^{1,2}_{0,0}(\mathbb{R}^{n+1}_+) $ and $ v=t^{-1}u\in\cal{D}^{1,2}_{2}(\mathbb{R}^{n+1}_+)$, it is easy to check that for all $\phi \in C_0^\infty(\mathbb{R}^{n+1}_+),$ it holds 
\begin{align*}
    \int_{\mathbb{R}^{n+1}_+} \nabla u\nabla \phi dydt=\int_{\mathbb{R}^{n+1}_+} t^2 \nabla v \nabla (t^{-1}\phi)dydt,
\end{align*} 
and 
\begin{align*}
    \int_{\mathbb{R}^{n+1}_+} t^{-2+\frac{n-1}{n+1}\beta}u^{1+\frac{2\beta}{n+1}}\phi dydt=\int_{\mathbb{R}^{n+1}_+} t^\beta v^{1+\frac{2\beta}{n+1}}(t^{-1}\phi) dydt.
\end{align*}
Using approximating and Proposition \ref{u-v-1}, we know that the above two equalities also hold for all $\phi\in \cal{D}_{0,0}^{1,2}(\mathbb{R}^{n+1}_+)$.
The Proposition is thus proved.
\hfill$\Box$

\begin{remark} 
It is worth pointing out that	inequality \eqref{sharp-1} also holds for  $n=0$. In fact,  it is %also known as {\color{red} Bliss's inequality} (see \cite{Bliss30}) 
  a  special case of Bliss Lemma \cite{Bliss30}.  It asserts that for $\beta>0,$ it holds  %with $q=2(1+\beta), p=2,  r=\beta$:
\begin{align*}
  \big(	\int_{0}^{\infty} t^{-2-\beta}|u|^{2(1+\beta)}dt \big)^{\frac{1}{\beta+1}} \leq C_\beta \int_{0}^{\infty} |u'(t)|^2 dt,\ \forall u\in C_0^\infty(\mathbb{R}_+)
\end{align*}
with the sharp constant
 $$C_\beta=\big(\frac{1}{1+\beta}\big)^{\frac{1}{1+\beta}}\big[\frac{\beta\Gamma(\frac{2+2\beta}{\beta})}{\Gamma(\frac{1}{\beta})\Gamma(\frac{1+2\beta}{\beta})}\big]^{\frac{\beta}{1+\beta}},$$ which is achieved by
 $$u(t)=\frac{Ct}{(1+At^\beta)^{1/\beta}},\ t\geq 0,$$
 for some positive constants $A$ and $C$.

\end{remark}

\bigskip

\subsubsection{Generalization}
%Inspired by the approach in Section 2.1.3, by 
Considering $u=t^\frac{\gamma}{2}v$ for $\gamma\in\mathbb{R}$, we have the following more general inequality equivalent to inequality \eqref{GGN-2-1}:
\begin{proposition}\label{prop2-9} Assume that $  \alpha\geq 1,\ \beta>-1,  \ \frac{n-1}{n+1}\beta\le\alpha\le\beta+2$ and   $\gamma\in\mathbb{R}$. %{\color{red}($\gamma>\alpha-1$ is not necessary)}.
	 Then for all $u\in C_0^\infty ({\mathbb{R}^{n+1}_+})$,
 \begin{align}\label{GGN-gamma}
&(\int_{\mathbb{R}^{n+1}_+ }t^{\beta-\frac{\gamma(n+\beta+1)}{n+\alpha-1}} |u|^{\frac{2(n+\beta+1)}{n+\alpha-1}}dydt)^{\frac{n+\alpha-1}{n+\beta+1}}\nonumber \\
\le &  S_{n+1, \alpha,\beta}^{-1} \int_{\mathbb{R}^{n+1}_+ }\big(t^{\alpha-\gamma}|\nabla u|^2 -\frac{\gamma(\gamma-2(\alpha-1)) }4\cdot t^{\alpha-\gamma-2} u^2 \big)dy dt,
\end{align}
 where  $S_{n+1, \alpha,\beta}$ is sharp.
 \end{proposition}

%For $\sigma>0$, set $\cal{D}^{1,2}_{\sigma,0}(\mathbb{R}^{n+1}_+)$ as the completion of  $C^\infty_0(\mathbb{R}^{n+1}_+)$ under the norm $\|\cdot\|_{\cal{D}^{1,2}_{\sigma}(\mathbb{R}^{n+1}_+)}$. 
Taking $\gamma=2(\alpha-1)\geq 0$ in Proposition \ref{prop2-9}, we have the following corollary, which apparently is more general than inequality \eqref{p-main-2}.
%(in the case $p=2$) as the special case ($\alpha=2$). 
\begin{corollary}\label{cor2-10}
Assume $ \alpha\geq 1,\beta>-1,$ and $ \frac{n-1}{n+1}\beta\le\ \alpha \le\beta+2.$ Then for all $u\in C_0^\infty(\mathbb{R}^{n+1}_+),$%$u\in \cal{D}^{1,2}_{2-\alpha,0} ({\mathbb{R}^{n+1}_+})$,
\begin{equation}\label{GGN-2-2}
(\int_{\mathbb{R}^{n+1}_+ }t^{\frac{n \beta-(\alpha-1)(2n+\beta+2)}{n+\alpha-1}} |u|^{\frac{2(n+\beta+1)}{n+\alpha-1}}dydt)^{\frac{n+\alpha-1}{n+\beta+1}} \le  S_{n+1, \alpha,\beta}^{-1} \int_{\mathbb{R}^{n+1}_+ }t^{2-\alpha}|\nabla u|^2dy dt,
\end{equation}
 where  $S_{n+1, \alpha,\beta}$ is sharp.
\end{corollary}

%As before, it is easy to check that  
%$$w \in \cal{D}^{1,2}_{2-\alpha,0}(\mathbb{R}^{n+1}_+) \text{ if and only if }    v=t^{1-\alpha}w\in \cal{D}^{1,2}_{\alpha}(\mathbb{R}^{n+1}_+), \text{ if } 1\leq \alpha\leq 2.$$

 We end up this subsection by discussing the relation between the inequality in Corollary \ref{cor2-6-p} and the one in Corollary \ref{cor2-10}. Generally, write $\alpha_1=2-\alpha<1$ and $\beta_1=\frac{n \beta-(\alpha-1)(2n+\beta+2)}{n+\alpha-1}$ in \eqref{GGN-2-2}, then
$$
\frac{2(n+\beta+1)}{n+\alpha-1}=
\frac{2(n+\beta_1+1)}{n+\alpha_1-1} \ \text{ if} \    \alpha\neq n+1 \, (\text{i.e. }\alpha_1\neq 1-n),$$
and \begin{align*}
\begin{cases}
\frac{n-1}{n+1}\beta_1\le\ \alpha_1 \le\beta_1+2, & \ \text{if } 1-n<\alpha_1<1,\\
\beta_1+2\le\ \alpha_1 \le \frac{n-1}{n+1}\beta_1,&\ \text{if } \alpha_1<1-n.
\end{cases}
\end{align*}
Thus letting $k=\frac{\alpha_1+\beta_1}{2}$ and $l=\beta_1$, we get that \eqref{GGN-2-2} coincides with \eqref{GGN-o-1-p} for $p=2$. Besides, the case $\alpha_1\geq 1$ naturally follows from inequality \eqref{GGN-2-1} since $C^\infty_0(\mathbb{R}^{n+1}_+)\subset C^\infty_0(\overline{\mathbb{R}^{n+1}_+})$.
For the remaining case $\alpha_1=1-n$,  we have $\alpha=1+n$ and $\beta_1=-n-1$. Then taking $s=\frac{n+\beta+1}{n}$, we get that \eqref{GGN-2-2} coincides with \eqref{GGN-o-1-n+1} for $p=2$. In conclusion, we have
\begin{corollary}\label{cor-alpha_1}

	\noindent (1) For $\alpha_1$ and $\beta_1$ satisfying
	\begin{align*}
	\begin{cases}
	\frac{n-1}{n+1}\beta_1\le\ \alpha_1 \le\beta_1+2, & \ \text{if } \alpha_1>1-n \text{ and }\alpha_1\neq 1,\\
		\frac{n-1}{n+1}\beta_1\le\ \alpha_1 <\beta_1+2, & \ \text{if } \alpha_1=1,\\
	\beta_1+2\le\ \alpha_1 \le \frac{n-1}{n+1}\beta_1&\ \text{if } \alpha_1<1-n,
	\end{cases}
	\end{align*} it holds
	\begin{equation*}
	(\int_{\mathbb{R}^{n+1}_+ }t^{\beta_1} |u|^{\frac{2(n+\beta_1+1)}{n+\alpha_1-1}}dydt)^{\frac{n+\alpha_1-1}{n+\beta_1+1}} \le  %C_{1, \alpha_1,\beta_1} 
	C\int_{\mathbb{R}^{n+1}_+ }t^{\alpha_1}|\nabla u|^2 dy dt,\ \forall u\in C^\infty_0(\mathbb{R}^{n+1}_+).
	\end{equation*}
	
		\noindent (2) For $\alpha_1=1-n$, $\beta_1=-n-1$, and  $s$ satisfying
		\begin{equation*}
		\begin{cases}
		s\in [2, \frac{2(n+1)}{n-1}],& \text{ if } n>1,\\
		s\ge 2 & \text{ if } n=1,
		\end{cases}
		\end{equation*}
		it holds
			\begin{align*}
		\big(\int_{\mathbb{R}^{n+1}_+}t^{-n-1}|u|^{s}dydt\big)^{\frac{2}{s}} \leq C\int_{\mathbb{R}^{n+1}_+}t^{1-n}|\nabla u|^2 dydt,\ \forall u\in C^\infty_0(\mathbb{R}^{n+1}_+).
		\end{align*}
\end{corollary}

%Condition $\frac{n-1}{n+1}\beta\le\ \alpha \le\beta+2$ implies $\beta>-1$. {\bf $\alpha=1$ is special, we also need to assume that $\beta>-1?$}

%Corollary \ref{cor2-10}(including $n=1$ case) is more general than Corollary \ref{cor2-6}.

\bigskip

\subsection{Hardy-Sobolev inequality on bounded domains}

In this subsection, we prove Theorem \ref{thm1-1}, that is, the Hardy-Sobolev inequality on  bounded domains with Lipschitz boundary. 

First, for $\Omega$ with $C^1$ boundary,  Theorem \ref{thm1-1} can be easily derived from next lemma which is usually referred to $\varepsilon$-level sharp inequality. And this lemma will be also used to  derive the existence result in the proof of Theorem \ref{extremal-attain} in  Section 3.
\begin{lemma}\label{var-level} Assume that $\Omega$ is  a bounded domain with  $C^1$ boundary and $\beta$ satisfies \eqref{beta-p}. Then for any $ \varepsilon>0$, there is $C(\varepsilon)>0$, such that for all $u\in W_0^{1,p} (\Omega)$,
\begin{equation}\label{maininequality-4}
(\int_{\Omega}\delta^{-p+\frac{n+1-p}{n+1}\beta} |u|^{p+\frac{p\beta}{n+1}}dx)^{\frac{n+1}{n+1+\beta}} \le  (C^*_{n+1, p, \beta}+\varepsilon) \int_{\Omega }|\nabla u|^p dx+C(\varepsilon)\int_{\Omega} |u|^p dx,
\end{equation}
where $C^*_{n+1, p, \beta}$ is the sharp constant in Theorem \ref{thm1-2}.
\end{lemma}

\begin{proof}
 Let $\{\Omega_i\}_{i=0}^{k}$ be an covering of $\Omega$, which satisfies that $\{\Omega_i\}_{i=1}^{k}$ is an covering of $\partial \Omega$, $\Omega_0\subset \Omega$ and  $dist(\Omega_0,\partial \Omega)>\delta_0>0$. Suppose $\{\eta_i^p\}_{i=0}^k$ is the partition of unity subordinate to $\{\Omega_i\}_{i=0}^k$, that is, $\eta_i$ satisfies
 $$\sum_{i=0}^k \eta_i^p=1 \text{ in } \Omega,\ 0\leq \eta_i\leq 1, \  \eta_i\in C_0^{\infty}(\Omega_i),\ i=0,1,\cdots,k.$$
By \eqref{beta-p}, we have that for $p<n+1$,
$$-p+\frac{n+1-p}{n+1}\beta\in [-p,0] \ \text{ and }\  p+\frac{p\beta}{n+1}\in [2,\frac{p(n+1)}{n+1-p}].$$
%and for $n=1$,
%$$-2+\frac{n-1}{n+1}\beta=0; \quad 2+\frac{2\beta}{n+1}\geq 2.$$
Then in $\Omega_0$, since $\delta(x)>\delta_0$, by H\"{o}lder inequality and Sobolev inequality, we have that for all $w\in W_0^{1,p}(\Omega_0)$,
\begin{align}\label{Omega_0}
\big( \int_{\Omega_0}\delta^{-p+\frac{n+1-p}{n+1}\beta} |w|^{p+\frac{p\beta}{n+1}}dx\big)^{\frac{n+1}{n+1+\beta}}\leq &\delta_0^{\frac{(n+1-p)\beta-p(n+1)}{n+1+\beta}} \|w\|_{L^{\frac{p(n+1+\beta)}{n+1}}(\Omega_0)}^p\nonumber\\
\leq & \delta_0^{\frac{(n+1-p)\beta-p(n+1)}{n+1+\beta}}
 \|w\|_{L^{\frac{p(n+1)}{n+1-p}}(\Omega_0)}^{p\theta}\|w\|_{L^p(\Omega_0)}^{p(1-\theta)}\nonumber\\
\leq & C \|\nabla w\|_{L^p(\Omega_0)}^{p\theta} \|w\|_{L^p(\Omega_0)}^{p(1-\theta)} \nonumber\\
	\leq &\varepsilon \|\nabla w\|_{L^p(\Omega_0)}^p + C(\varepsilon) \|w\|_{L^p(\Omega_0)}^p,
\end{align}
where $\theta=\frac{(n+1)\beta}{p(n+1+\beta)}$.
	It is  easy to check that \eqref{Omega_0} also holds for the case $p=n+1$ by H\"{o}lder inequality and Sobolev embedding of $W_{0}^{1,n+1}(\Omega_0)$.

Since $\Omega$ has $C^1$ boundary,  we can assume $\Omega_i (i=1,\cdots, k)$ is small enough, such that there is  a $C^1$ map $\psi_i$ satisfying
\begin{align*}
	\psi_i(\Omega_i\cap \Omega)=U_i\subset\mathbb{R}^{n+1}_+, \ \psi_i(\Omega_i\cap \partial \Omega)\subset \partial \mathbb{R}^{n+1}_+,
\end{align*}
and for $\varepsilon\ll 1$, if we write $(y,t)=\psi_i(x)$, then 
\begin{align}\label{flat}
	\frac{dydt}{(1+\varepsilon)^{n+1}}\leq dx\leq (1+\varepsilon)^{n+1}dydt, \ \frac{t}{1+\varepsilon}\leq \delta(x) \leq (1+\varepsilon)t, \text{ for }x\in \Omega_i.
\end{align}
%{\color{red} (check Lipschitz boundary condition)} 
Therefore, by   inequality \eqref{p-main-2} on the upper half space, for $w\in W_0^{1,p}(\Omega_i\cap\Omega)(i=1,\cdots,k)$, we have that
\begin{align}\label{Omega_i}
(\int_{\Omega_i\cap\Omega}\delta^{-p+\frac{n+1-p}{n+1}\beta} |w|^{p+\frac{p\beta}{n+1}}dx)^{\frac{n+1}{n+1+\beta}} \leq& (1+\varepsilon)^{\theta_0} (\int_{U_i}t^{-p+\frac{n+1-p}{n+1}\beta} |w\circ \psi^{-1}|^{p+\frac{p\beta}{n+1}}dydt)^{\frac{n+1}{n+1+\beta}}\nonumber\\
\leq& C^*_{n+1, p, \beta}(1+\varepsilon)^{\theta_1}\int_{U_i}|\nabla (w\circ\psi_i^{-1})|^p dydt\nonumber\\
\leq & C^*_{n+1, p, \beta}(1+\varepsilon)^{\theta_2}\int_{\Omega_i\cap\Omega}|\nabla w|^p dx,
\end{align}
for some positive numbers $\theta_0,\theta_1,\theta_2$. We  rewrite $C^*_{n+1,p,\beta}(1+\varepsilon)^{\theta_2}$ as $C^*_{n+1,p,\beta}+\varepsilon$. Since $\sum_{i=0}^k |\eta_i|^p=1$ in $\Omega$, applying Minkowski inequality, \eqref{Omega_0} and \eqref{Omega_i}, we have, for $\varepsilon\ll 1$ and any $u\in W_0^{1,p}(\Omega)$,  that
\begin{align}\label{sum}
&(\int_{\Omega}\delta^{-p+\frac{n+1-p}{n+1}\beta} |u|^{p+\frac{p\beta}{n+1}}dx)^{\frac{n+1}{n+1+\beta}} \nonumber\\
=&(\int_{\Omega}\delta^{-p+\frac{n+1-p}{n+1}\beta} |\sum_{i=0}^{k}\eta_i^p u^p|^{\frac{n+1+\beta}{n+1}}dx)^{\frac{n+1}{n+1+\beta}}\nonumber\\
\leq & \sum_{i=0}^{k}(\int_{\Omega_i\cap\Omega}\delta^{-p+\frac{n+1-p}{n+1}\beta} |\eta_i^p u^p|^{\frac{n+1+\beta}{n+1}}dx)^{\frac{n+1}{n+1+\beta}}\nonumber\\
\leq & \sum_{i=0}^{k} \big[(C^*_{n+1,p,\beta}+\varepsilon)\int_{\Omega_i\cap\Omega}|\nabla (\eta_i u)|^p dx \big]+ C(\varepsilon)\int_{\Omega_0}| (\eta_i u)|^p dx\nonumber\\
\leq & (C^*_{n+1,p, \beta}+\varepsilon)\int_{\Omega}\sum_{i=0}^k |\nabla (\eta_i u)|^p dx+ C(\varepsilon)\int_{\Omega}|u|^p dx.
%\leq& (\mu_{n+1,\tau}^{-1}+\varepsilon)\int_{\Omega}|\nabla  v|^2 dx + C(\varepsilon)\int_{\Omega}|  v|^2 dx
 \end{align}
 Since $$\sum_{i=0}^k |\nabla (\eta_i u)|^p\leq \sum_{i=0}^k\big((1+\varepsilon)\eta_i^p|\nabla u|^p +C(\varepsilon)|\nabla \eta_i|^p u^p\big)\leq (1+\varepsilon)|\nabla u|^p+C(\varepsilon)|u|^p, $$
bringing this  back to \eqref{sum}, we get \eqref{maininequality-4}.
\end{proof}

\medskip

 For $\Omega$ bounded with Lipschitz boundary, the constants in the two formulas of \eqref{flat} may not be as small as $1+\varepsilon$ at the same time. However, we can replace  them by constants depending on $\Omega$ due to the Lipschitz boundary condition, which can also be used to derive Theorem \ref{thm1-1}.

\section{Sharp constants on bounded domains.}

In this section, we consider the sharp constant of Hardy-Sobolev inequality on a bounded domain $\Omega$ with Lipschitz boundary, and give the proofs of Proposition \ref{leq}, Proposition \ref{extremal-attain},   and Proposition \ref{ball-equality}.
 
 \medskip
 
\noindent{\bf Proof of Proposition \ref{leq}.} By Theorem \ref{thm1-1}, we know that $\mu_{n+1,\beta}(\Omega)>0$. Next, we prove $\mu_{n+1,\beta}(\Omega)\leq \mu^*_{n+1,\beta}$.

Since $\Omega$ has Lipschitz boundary, after suitable translation and rotation, we can assume that $0\in \partial \Omega$, $\partial\mathbb{R}^{n+1}_+$ is the tangent hyperplane to $\Omega$ at $0$, and for any $A>0$, there is $h>0$, such that 
$$K_{A,h}:=\{x=(y,t):\ y\in\mathbb{R}^n, 0<t<h, |y|\leq At \}\subset\Omega.$$
Fix $A$ and $h$. By the definition of $\mu^*_{n+1,\beta}$, for any $\varepsilon>0$, there is $u\in C^\infty_0(\mathbb{R}^{n+1}_+)$, such that 
\begin{align*}
    \mu^*_{n+1,\beta}\leq J_{n+1,\beta,\mathbb{R}^{n+1}_+}[u]<\mu^*_{n+1,\beta}+\varepsilon.
\end{align*}
Take $\tilde{u}(y,t)=\lambda^{\frac{n-1}{2}}u(\lambda y, \lambda t),$ where $\lambda>0$ is large enough, such that 
$$\text{supp} \tilde{u}\subset K_{A,h}\subset\Omega,$$
and 
\begin{align}\label{dis-leq}
\delta_\Omega(x)<(1+\varepsilon)t,\quad\forall x=(y,t)\in \text{supp} \tilde{u}.
\end{align}
By scaling invariance, we have that $$J_{n+1,\beta,\mathbb{R}^{n+1}_+}[\tilde{u}]=J_{n+1,\beta,\mathbb{R}^{n+1}_+}[u]<\mu^*_{n+1,\beta}+\varepsilon.$$
Since $-2+\frac{n-1}{n+1}\beta\leq 0$, by \eqref{dis-leq},
we have that 
$$\mu_{n+1,\beta,\Omega}\leq J_{n+1,\beta,\Omega}[\tilde{u}]\leq\frac{J_{n+1,\beta,\mathbb{R}^{n+1}_+}[\tilde{u}]}{(1+\varepsilon)^{\frac{n+1}{n+\beta+1}(-2+\frac{n-1}{n+1}\beta)}}<\frac{\mu^*_{n+1,\beta}+\varepsilon}{(1+\varepsilon)^{\frac{n+1}{n+\beta+1}(-2+\frac{n-1}{n+1}\beta)}}.$$
Letting $\varepsilon\to 0^+$, we get $\mu_{n+1,\beta}(\Omega)\leq \mu^*_{n+1,\beta}$.
\hfill$\Box$
\smallskip

\begin{remark}\label{rmk-3-1}
In fact, for general domain $\Omega$,
% (no need to be bounded with Lipschitz boundary), 
if $\partial \Omega$ possesses a tangent plane at least at one point, then $\mu_{n+1,\beta}(\Omega)\leq\mu_{n+1,\beta}^*$. See \cite{Davies95} or \cite{MMP98} for $\beta=0$.
\end{remark}

Next, we give the  sufficient condition for $\mu_{n+1,\beta}(\Omega)$ being achieved  in $W_0^{1,2}(\Omega).$

\noindent{\bf Proof of Proposition \ref{extremal-attain}.}
We apply Lemma \ref{var-level} to show that if $0<\mu_{n+1,\beta}(\Omega)<\mu_{n+1,\beta}^*$, then $\mu_{n+1,\beta}(\Omega)$ is achieved.
	Suppose $\{u_j\}\subset C_0^\infty(\Omega)$ is a normalized nonnegative minimizing sequence of $\mu_{n+1,\beta}(\Omega)$, that is,
	\begin{align*}
\int_{\Omega}\delta^{-2+\frac{n-1}{n+1}\beta} |u_j|^{2+\frac{2\beta}{n+1}}dx=1, \quad \lim_{j\to \infty }\int_{\Omega }|\nabla u_j|^2dx=\mu_{n+1,\beta}(\Omega).
	\end{align*}
	Thus there is $u\in W^{1,2}_0(\Omega)$, such that
	\begin{align*}
	&u_j\rightharpoonup u \text{ weakly in } W^{1,2}_0(\Omega),\\
	&u_j\to u \text{ strongly in } L^2(\Omega), \\
	&u_j \to u \text{ a.e. in }\Omega.
	\end{align*}
	Then we can get that
	\begin{align}\label{B-L-1}
	\mu_{n+1,\beta}(\Omega)=\int_{\Omega }|\nabla u_j|^2dx+o(1)=\int_{\Omega }|\nabla u_j-\nabla u|^2dx+\int_{\Omega }|\nabla u|^2dx+o(1).
	\end{align}
	By  Brezis-Lieb lemma (see \cite{BL83}), it holds
	\begin{align*}
		\lim_{j\to\infty}\big(\int_{\Omega}\delta^{-2+\frac{n-1}{n+1}\beta}|u_j|^{2+\frac{2\beta}{n+1}}dx-\int_{\Omega}\delta^{-2+\frac{n-1}{n+1}\beta} |u_j-u|^{2+\frac{2\beta}{n+1}}dx\big)=\int_{\Omega}\delta^{-2+\frac{n-1}{n+1}\beta}|u|^{2+\frac{2\beta}{n+1}}dx.
	\end{align*}
	So we have
	\begin{align}\label{B-L-2}
		1=&\Big(\int_{\Omega}\delta^{-2+\frac{n-1}{n+1}\beta} |u_j|^{2+\frac{2\beta}{n+1}}dx\Big)^{\frac{n+1}{n+1+\beta}}\nonumber\\
		=&\Big(\int_{\Omega}\delta^{-2+\frac{n-1}{n+1}\beta} |u_j-u|^{2+\frac{2\beta}{n+1}}dx+\int_{\Omega}\delta^{-2+\frac{n-1}{n+1}\beta} |u|^{2+\frac{2\beta}{n+1}}dx+o(1)\Big)^{\frac{n+1}{n+1+\beta}}\nonumber\\
		\leq&\big(\int_{\Omega}\delta^{-2+\frac{n-1}{n+1}\beta} |u_j-u|^{2+\frac{2\beta}{n+1}}dx\big)^{\frac{n+1}{n+1+\beta}}+\big(\int_{\Omega}\delta^{-2+\frac{n-1}{n+1}\beta} |u|^{2+\frac{2\beta}{n+1}}dx\big)^{\frac{n+1}{n+1+\beta}}+o(1).
	\end{align}
%	$$\int_{\Omega}\delta^{-2+\frac{n-1}{n+1}\beta}|v|^{2+\frac{2\beta}{n+1}}dx\leq 1,\ \int_{\Omega}\delta^{-2+\frac{n-1}{n+1}\beta}|v_j-v|^{2+\frac{2\beta}{n+1}}dx\leq 1,$$
Since $0<\mu_{n+1,\beta}(\Omega)<\mu^*_{n+1,\beta}$, we can choose $\varepsilon>0$ small enough, such that $$\mu_{n+1,\beta}^{-1}(\Omega)>(\mu_{n+1,\beta}^{*})^{-1}+\varepsilon.$$
In \eqref{B-L-2}, using Lemma  \ref{var-level} to $u_j-u$ with $p=2$  (noticing that $C^*_{n+1,2,\beta}=(\mu_{n+1,\beta}^{*})^{-1}$),  then using \eqref{B-L-1} and noticing $u_j\to u$ in $L^2(\Omega)$, we have
\begin{align*}
	1\leq& ((\mu_{n+1,\beta}^{*})^{-1}+\varepsilon) \int_{\Omega }|\nabla u_j-\nabla u|^2dx+C(\varepsilon)\int_{\Omega} |u_j-u|^2 dx\\
	&+\mu_{n+1,\beta}^{-1}(\Omega)\int_{\Omega }|\nabla u|^2dx+o(1)\\
	\leq &1-\big(\mu_{n+1,\beta}^{-1}(\Omega)-(\mu_{n+1,\beta}^{*})^{-1}-\varepsilon\big) \int_{\Omega }|\nabla u_j-\nabla u|^2dx+o(1),
\end{align*}
which implies 
$$\int_{\Omega }|\nabla u_j-\nabla u|^2dx\to 0,\ \text{as }j\to \infty,$$
as well as
$$\int_{\Omega}\delta^{-2+\frac{n-1}{n+1}\beta} |u_j-u|^{2+\frac{2\beta}{n+1}}dx\to 0, \ \text{as }j\to \infty.$$
Hence we obtain that $u\in W^{1,2}_0(\Omega)$ is the minimizer of $\mu_{n+1,\beta}(\Omega)$ with $$\int_{\Omega }|\nabla u|^2dx=\mu_{n+1,\beta}(\Omega),\ \text{and }\int_{\Omega}\delta^{-2+\frac{n-1}{n+1}\beta} |u|^{2+\frac{2\beta}{n+1}}dx=1.$$
\hfill$\Box$

\bigskip

%{\color{blue} Next, we consider the sharp constant of HS inequality on balls. We regard the distance function in the  HS inequality as weight function.   We notice that after a suitable Kelvin transformation (for example, see \eqref{transform-2}), the HS inequality on the upper half space becomes a new inequality on some ball $B$, which only differs from the HS inequality on $B$ with the weight functions. Then  by comparing the two weight functions, we can obtain the relation between  constants on balls and on the upper half space and prove  Theorem \ref{ball-equality}.}

%\textbf{I would rather to say that using the stereo-graphic projection, we can obtain the HS on a unite ball (put (3.5), (3.6) into the sharp inequality. Re-write this part please. 3-23-2022}

Finally, after proper Kelvin transformation, we can obtain a sharp weighted inequality on a ball (see \eqref{inequ-ball} below), which is equivalent to inequality \eqref{sharp-1}. By comparing the weight function in \eqref{inequ-ball} with the distance function in Hardy-Sobolev inequality, the relation between sharp constants of Hardy-Sobolev inequality on balls and on the upper half space, i.e. Proposition \ref{ball-equality}, can be easily obtained.
%\noindent{\bf Proof of Proposition \ref{ball-equality}.}

Set $e_{n+1}=(0,\cdots,0,1)\in \mathbb{R}^{n+1}$. %By \eqref{dilation}, we know that
%\begin{align}\label{mu-invar}
%	\mu_{n+1,\beta}(B_1(0))=\mu_{n+1,\beta}(B_{\frac12}(-\frac{e_{n+1}}{2})).
%\end{align} 
%Then we only need to show that $\mu_{n+1,\beta}(B_{\frac12}(-\frac{e_{n+1}}{2}))=\mu_{n+1,\beta}^*,$ and $\mu_{n+1,\beta}(B_{\frac12}(-\frac{e_{n+1}}{2}))$ is not achieved in $W_0^{1,2}(B_{\frac12}(-\frac{e_{n+1}}{2}))$.
Consider the reflection with respect to $\partial B_1(-e_{n+1})$ as
\begin{align}\label{transform-1}
x:=(x', x_{n+1})=-e_{n+1}+\frac{(y,t)+e_{n+1}}{|(y,t)+e_{n+1}|^2} .
\end{align} 
This projects $\mathbb{R}^{n+1}_+$ to $ B_{\frac12}(-\frac{e_{n+1}}{2})$ and 
$\partial \mathbb{R}^{n+1}_+$ to $\partial B_{\frac12}(-\frac{e_{n+1}}{2})$. 
For any $u\in \cal{D}_{0,0}^{1,2}(\mathbb{R}^{n+1}_+)$, set
the Kelvin transformation of $u$ with respect to $\partial B_1(-e_{n+1})$ as
\begin{align}\label{transform-2}
\psi(x)=\frac{1}{|x+e_{n+1}|^{n-1}}u(-e_{n+1}+\frac{x+e_{n+1}}{|x+e_{n+1}|^2}),\quad x\in B_{\frac12}(-\frac{e_{n+1}}{2}).
\end{align}
By simple calculations, we have 
\begin{align}\label{1}
	\int_{\mathbb{R}^{n+1}_+}|\nabla u|^2 dydt=\int_{B_{\frac12}(-\frac{e_{n+1}}{2})}|\nabla \psi|^2 dx
\end{align}
and 
\begin{align}\label{2}
\int_{\mathbb{R}^{n+1}_+}t^{-2+\frac{n-1}{n+1}\beta}|u|^{2+\frac{2\beta}{n+1}} dydt=\int_{B_{\frac12}(-\frac{e_{n+1}}{2})}\big(\frac14-|x+\frac{ e_{n+1}}{2}|^2\big)^{-2+\frac{n-1}{n+1}\beta}| \psi|^{2+\frac{2\beta}{n+1}} dx.
\end{align}
It is easy to check that \eqref{transform-2} gives a  bijection from $\cal{D}_{0,0}^{1,2}(\mathbb{R}^{n+1}_+)$ to $ W_0^{1,2}(B_{\frac12}(-\frac{e_{n+1}}{2}))$. Then by \eqref{1}, \eqref{2} and sharp inequality \eqref{sharp-1}, we get a sharp inequality on $B_{\frac12}(-\frac{e_{n+1}}{2})$: for any $ \psi\in W_0^{1,2}(B_{\frac12}(-\frac{e_{n+1}}{2}))$,
\begin{align}\label{inequ-ball}
    \Big(\int_{B_{\frac12}(-\frac{e_{n+1}}{2})} & \big(\frac14-|x+\frac{ e_{n+1}}{2}|^2\big)^{-2+\frac{n-1}{n+1}\beta}| \psi|^{2+\frac{2\beta}{n+1}} dx\Big)^{\frac{n+1}{n+\beta+1}} \nonumber \\
    & \leq (\mu_{n+1,\beta}^*)^{-1}\int_{B_{\frac12}(-\frac{e_{n+1}}{2})}|\nabla \psi|^2 dx.
\end{align}

We now apply \eqref{inequ-ball} to prove Proposition \ref{ball-equality}.

\medskip{}

\noindent{\bf Proof of Proposition \ref{ball-equality}.} %{\color{blue} Notice that the transform \eqref{transform-1} projects the distance function $\delta_{\mathbb{R}^{n+1}_+}=t$ to $\frac14-|x+\frac{ e_{n+1}}{2}|^2$.} %
In \eqref{inequ-ball}, we have, 
for $x\in B_{\frac12}(-\frac{e_{n+1}}{2}),$
 \begin{align}\label{3}\frac14-|x+\frac{ e_{n+1}}{2}|^2=\big(\frac12+|x+\frac{ e_{n+1}}{2}|\big)\big(\frac12-|x+\frac{ e_{n+1}}{2}|\big)<&\frac12-|x+\frac{ e_{n+1}}{2}| \nonumber \\
 =&\delta_{B_{\frac12}(-\frac{e_{n+1}}{2})}(x).
 \end{align}
 For $\beta$ satisfying \eqref{beta-0}, we have  $-2+\frac{n-1}{n+1}\beta\leq 0$, then combining 
 %\eqref{1}, \eqref{2} and \eqref{3}, 
 \eqref{inequ-ball} and \eqref{3}, we have, for any $ \psi\in W_0^{1,2}(B_{\frac12}(-\frac{e_{n+1}}{2}))$,
 %\begin{align*}
 %\frac{\int_{\mathbb{R}^{n+1}_+}|\nabla u|^2 dydt}{\Big(\int_{\mathbb{R}^{n+1}_+}t^{-2+\frac{n-1}{n+1}\beta}|u|^{2+\frac{2\beta}{n+1}} dydt\Big)^{\frac{n+1}{n+\beta+1}}}
 %\leq\frac{\int_{B_{\frac12}(-\frac{e_{n+1}}{2})}|\nabla \psi|^2 dx}
 %{\Big(\int_{B_{\frac12}(-\frac{e_{n+1}}{2})} \delta_{B_{\frac12}(-\frac{e_{n+1}}{2})}^{-2+\frac{n-1}{n+1}\beta}(x)| \psi|^{2+\frac{2\beta}{n+1}} dx\Big)^{\frac{n+1}{n+\beta+1}}},
 %\end{align*}
 \begin{align}\label{HS-ball}
    \Big(\int_{B_{\frac12}(-\frac{e_{n+1}}{2})}\delta_{B_{\frac12}(-\frac{e_{n+1}}{2})}^{-2+\frac{n-1}{n+1}\beta}(x)| \psi|^{2+\frac{2\beta}{n+1}} dx\Big)^{\frac{n+1}{n+\beta+1}}\leq (\mu_{n+1,\beta}^*)^{-1}\int_{B_{\frac12}(-\frac{e_{n+1}}{2})}|\nabla \psi|^2 dx,
\end{align}
% Taking infimum in $u\in \cal{D}_{0,0}^{1,2}(\mathbb{R}^{n+1}_+)\backslash\{0\}$ and $\psi\in W_0^{1,2}(B_{\frac12}(-\frac{e_{n+1}}{2}))\backslash\{0\}$, respectively, we have 
which means $$\mu_{n+1,\beta}^*\leq \mu_{n+1,\beta}(B_{\frac12}(-\frac{e_{n+1}}{2})).$$
Then by Proposition \ref{leq}, we obtain the equality:
 $$\mu_{n+1,\beta}(B_{\frac12}(-\frac{e_{n+1}}{2}))=\mu_{n+1,\beta}^*.$$

 Moreover, if $u$ and $\psi$ satisfy \eqref{transform-2}, it holds 
\begin{align}\label{compare}
   J_{n+1,\beta,\mathbb{R}^{n+1}_+}[u]\leq J_{n+1,\beta,B_{\frac12}(-\frac{e_{n+1}}{2})}[\psi].
 \end{align}
 For $\beta=\frac{2(n+1)}{n-1}$ with $n\geq 2$, $\mu_{n+1,\frac{2(n+1)}{n-1}}(B_{\frac12}(-\frac{e_{n+1}}{2}))$ cannot be achieved in  $W_0^{1,2}(B_{\frac12}(-\frac{e_{n+1}}{2}))$ since  $\mu_{n+1,\frac{2(n+1)}{n-1}}^*$ is not achieved in $\cal{D}_{0,0}^{1,2}(\mathbb{R}^{n+1}_+)$. For $\beta$ satisfying \eqref{beta-strict} or $\beta=0$, since $-2+\frac{n-1}{n+1}\beta<0$,   in \eqref{compare}, the strict inequlity holds, then $\mu_{n+1,\beta}(B_{\frac12}(-\frac{e_{n+1}}{2}))$ cannot be achieved in $W_{0}^{1,2}(B_{\frac12}(-\frac{e_{n+1}}{2}))$ no matter   $\mu_{n+1,\beta}^*$ is  achieved
 in $\cal{D}^{1,2}_{0,0}(\mathbb{R}^{n+1}_+)$ or not.   
 
\hfill$\Box$

\medskip

	 Similarly, if we consider exterior domain of a ball, for example, $B^c_{\frac12}(-\frac{e_{n+1}}{2})$, we can get that for $\beta$ satisfying \eqref{beta-strict}, 
		\begin{align}\label{HS-outball}
		0\leq\mu_{n+1,\beta}(B^c_{\frac12}(-\frac{e_{n+1}}{2}))<\mu_{n+1,\beta}^*.
		\end{align}
%		However, we don not know whether $\mu_{n+1,\beta}(B^c_{\frac12}(-\frac{e_{n+1}}{2}))$ is positive or not. But for $\beta=0$, it was given in \cite{MMP98} that $\mu_{n+1,0}(B^c_{\frac12}(-\frac{e_{n+1}}{2}))=0$ for $n=1$ and $\mu_{n+1,0}(B^c_{\frac12}(-\frac{e_{n+1}}{2}))=\mu_{n+1,0}^*=\frac14$ for $n\geq 2$. 
%	
   %To prove \eqref{HS-outball}, 
   In fact, set Kelvin transformation of $\psi$ with respect to $\partial B_{\frac12}(-\frac{e_{n+1}}{2})$, that is, for $x\in B_{\frac12}(-\frac{e_{n+1}}{2})$,
   $$z=-\frac12 e_{n+1}+\frac{\frac14 (x+\frac12 e_{n+1})}{|x+\frac12 e_{n+1}|^2},\ \text{ and } \tilde\psi(z)=(\frac{\frac12}{|z+\frac12 e_{n+1}|})^{n-1}\psi(x).$$
   This reflection projects $B_{\frac12}(-\frac{e_{n+1}}{2})$ to $B^c_{\frac12}(-\frac{e_{n+1}}{2})$ and by \eqref{1} and \eqref{2}, it is easy to check that  %$\frac14-|x+\frac{ e_{n+1}}{2}|^2$ to $|z+\frac12 e_{n+1}|^2-\frac14$. 
   \begin{align*}
      	\int_{\mathbb{R}^{n+1}_+}|\nabla u|^2 dydt=\int_{B^c_{\frac12}(-\frac{e_{n+1}}{2})}|\nabla \tilde\psi|^2 dz
\end{align*}
and 
\begin{align*}
\int_{\mathbb{R}^{n+1}_+}t^{-2+\frac{n-1}{n+1}\beta}|u|^{2+\frac{2\beta}{n+1}} dydt=\int_{B^c_{\frac12}(-\frac{e_{n+1}}{2})}\big(|z+\frac{ e_{n+1}}{2}|^2-\frac14\big)^{-2+\frac{n-1}{n+1}\beta}| \tilde\psi|^{2+\frac{2\beta}{n+1}} dz.
\end{align*}
   Similarly to \eqref{3}, we have that
   $$|z+\frac12 e_{n+1}|^2-\frac14>\delta_{B^c_{\frac12}(-\frac{e_{n+1}}{2})}(z), \ \text{for }\ z\in B^c_{\frac12}(-\frac{e_{n+1}}{2}).$$
   Then %by similar calculation to \eqref{1} and \eqref{2}, 
   we can obtain that for $-2+\frac{n-1}{n+1}\beta<0$,
   \begin{align*}
   J_{n+1,\beta,B^c_{\frac12}(-\frac{e_{n+1}}{2})}[\tilde{\psi}]<J_{n+1,\beta,\mathbb{R}^{n+1}_+}[u].
   \end{align*}
   Since $\mu_{n+1,\beta}^*$ is achieved for $\beta$ satisfying \eqref{beta-strict}, we can obtain \eqref{HS-outball}.

\section{Some examples}

	In this section, we discuss the sharp constant of Hardy-Sobolev inequality for some specific domains, including some  domains with non-Lipschitz boundary point and some unbounded domains.
	% give some examples of $\Omega$, for which $\mu_{n+1,\beta}(\Omega)=0$ or $\mu_{n+1,\beta}(\Omega)\in (0, \mu_{n+1,\beta}^*)$.
%	\subsection{$n=1$}

Before introducing the examples, we first give a generalization of Lemma 12 in \cite{MMP98}, which is Lemma \ref{limit} below and  useful in the following examples. 

\begin{definition}
	Let $\Omega$ be a domain in $\mathbb{R}^{n+1}.$ A sequence of domains $\{\Omega_k\}$ is said to be a normal approximation sequence for $\Omega$, if it satisfies the following two conditions:
	\begin{align*}
	\delta_{\Omega_k}(x)\to \delta_\Omega(x),\quad \forall x\in\Omega,
	\end{align*}
	and for every compact subset $K$ of $\Omega$, there is an integer $j$, such that 
	\begin{align*}
	K\subset\cap_{k=j}^{\infty}\Omega_k.
	\end{align*}
\end{definition}
\begin{lemma}\label{limit}
	Assume that $\Omega$ is a domain in $\mathbb{R}^{n+1}$ and $\{\Omega_k\}$ is a normal approximating sequence for $\Omega$. Then for  $\beta$ satisfying \eqref{beta-0}, it holds
	\begin{align}\label{limit-1}
	\varlimsup_{k\to \infty}\mu_{n+1,\beta}(\Omega_k)\leq \mu_{n+1,\beta}(\Omega).
	\end{align}
\end{lemma}
The proof of Lemma \ref{limit} is similar to Lemma 12 in \cite{MMP98}.
	
	\medskip
	
	\noindent{\bf Example 1.} The Punctured Space: Let $\mathbb{R}^{n+1}_*=\mathbb{R}^{n+1}\backslash \{0\}$, then for $n\geq 2$, the inequality is the classical Hardy-Sobolev inequality \eqref{HS-1} with  $p=2$. We rewrite it with parameter $\beta$: 
	\begin{align}\label{HS-2}
	\big(\int_{\mathbb{R}^{n+1}_*}|x|^{-2+\frac{n-1}{n+1}\beta}|u|^{2+\frac{2\beta}{n+1}}dx\big)^{\frac{n+1}{n+\beta+1}} \leq \mu_{n+1,\beta}^{-1}(\mathbb{R}^{n+1}_*)\int_{\mathbb{R}^{n+1}_*}|\nabla u|^2 dx , \forall u\in \cal{D}^{1,2}_0(\mathbb{R}^{n+1}),
	\end{align}
	with $\beta\in [0,\frac{2(n+1)}{n-1}]$.
	Note that $C^\infty_0(\mathbb{R}^{n+1}\backslash\{0\})$ is dense in $\cal{D}_0^{1,2}(\mathbb{R}^{n+1})$ for $n\geq 2$, while it fails for $n=1$.
	For $\beta\in (0,\frac{2(n+1)}{n-1}]$ with $n\geq 2$, %The study in \cite{CW01,CC93,ZZ15} gave 
	the sharp constant is
	%%	\mu_{n+1,\beta}(\mathbb{R}^{n+1}_*)=(n-1)^{\frac{2(n+1)+\beta}{n+\beta+1}}\big(\frac{n+\beta+1}{n+1}\big)^{\frac{n+1}{n+\beta+1}}\Big(\frac{2\pi^{\frac{n+1}{2}}\Gamma(\frac{n+1}{\beta}+1)\Gamma(\frac{n+1}{\beta}+2)}{\Gamma(\frac{n+1}{2})\Gamma(\frac{2(n+1}{\beta}+2)}\Big)^{\frac{\beta}{n+\beta+1}},
	\begin{align}\label{constant-HS}
	\mu_{n+1,\beta}(\mathbb{R}^{n+1}_*)=(n+\beta+1)(n-1)\big(\frac{n-1}{n+1}\big)^{\frac{n+1}{n+\beta+1}}\Big(\frac{2\pi^{\frac{n+1}{2}}\Gamma^2(\frac{n+\beta+1}{\beta})}{\beta\Gamma(\frac{n+1}{2})\Gamma(\frac{2(n+\beta+1)}{\beta}+2)}\Big)^{\frac{\beta}{n+\beta+1}},
	\end{align}
	and the extremal functions are of the form 
	\begin{align}\label{minimizer-HS}
	u(x)=\frac{C}{\big(A+|x|^{\frac{\beta(n-1)}{n+1}}\big)^{\frac{n+1}{\beta}}},\quad x\in\mathbb{R}^{n+1},
	\end{align}
	for $C>0$ and $A>0$. See, for example \cite{Lieb83,CC93}. For $\beta=0$, inequality \eqref{HS-2} is the $({n+1})$-dimensional Hardy inequality with $\mu_{n+1,0}(\mathbb{R}^{n+1}_*)=(\frac{n-1}{2})^2$, but the equality does not hold for any nontrivial function in $\cal{D}^{1,2}_0(\mathbb{R}^{n+1}).$
	
	\smallskip
	
	For $n=1$, inequality \eqref{HS-2} does not hold, that is, for $\beta\geq0$, 
	$$\mu_{2,\beta}(\mathbb{R}^{2}_*)=0.$$ In fact, 
	take $\eta_R\in C_0^\infty(-\infty,+\infty)$ for $R>0$, such that 
	$$\eta_R= 1 \text{ in } (-R, R),\ \eta_R=0 \text{ in } (-\infty, -2R)\cup (2R, +\infty),\ \text{and }|\eta'_R|\leq \frac{A}{R}.$$ 
	Let $u_R(x)=\eta_R(\ln |x|)$, then $u_R\in C_0^\infty(\mathbb{R}^2_*)$.
	It is easy to check that
	 \begin{align*}
	J_{2,\beta,\mathbb{R}^2_*}[u_R]\leq \frac{(4\pi)^{\frac{\beta}{2+\beta}}A^2}{R^{\frac{4+\beta}{2+\beta}}}.
	\end{align*}
	Sending $R\to +\infty$, we get $\mu_{2,\beta}(\mathbb{R}^2_*)=0$ for $\beta\geq 0$.

	\medskip

	\noindent{\bf Example 2.} Exterior Domain: Let $\Omega\subset\mathbb{R}^{n+1}$ be an exterior domain such that $\partial \Omega$ possesses a tangent plane at least at one point, then by Remark \ref{rmk-3-1},  $\mu_{n+1,\beta}(\Omega)\leq\mu_{n+1,\beta}^*$ for $\beta$ satisfying \eqref{beta-0}. In particular, for $\Omega=B_1^c(0)$, we obtain from Section 3 that $\mu_{n+1,\beta}(B_1^c(0))<\mu_{n+1,\beta}^*$ for $\beta$ satisfying \eqref{beta-strict}. As for $\beta=0$, it was given in \cite{MMP98} that $\mu_{n+1,0}(B^c_{\frac12}(-\frac{e_{n+1}}{2}))=0$ for $n=1$ and $\mu_{n+1,0}(B^c_{\frac12}(-\frac{e_{n+1}}{2}))=\mu_{n+1,0}^*=\frac14$ for $n\geq 2$. 
	
	On the other hand, assume $0\notin \overline{\Omega}$ and consider $\Omega_k=\frac{1}{k}\Omega$, then $\{\Omega_k\}$ is a normal approximating sequence for $ \mathbb{R}^{n+1}_*$ as $k\to +\infty$ and  $\mu_{n+1,\beta}(\Omega_k)=\mu_{n+1,\beta}(\Omega)$. Then according to Lemma \ref{limit}, $\mu_{n+1,\beta}(\Omega)\leq \mu_{n+1,\beta}(\mathbb{R}^{n+1}_*)$. In conclusion, 
	\begin{align*}
		\mu_{n+1,\beta}(\Omega)\leq \min\{\mu_{n+1,\beta}(\mathbb{R}^{n+1}_*),  \mu_{n+1,\beta}^*\}.
	\end{align*}
	In particular, for $n=1$ and $\beta\geq 0$, $\mu_{2,\beta}(\mathbb{R}^{2}_*)=0$ implies $\mu_{2,\beta}(\Omega)=0$. Also notice that, by \eqref{constant-HS} and the sharp constants in Theorem \ref{thm1-2}, we have $\mu_{n+1,1}(\mathbb{R}^{n+1}_*)<\mu_{n+1,1}^*$ for $\beta=1$, $n\leq 3$, and  $\mu_{n+1,2}(\mathbb{R}^{n+1}_*)<\mu_{n+1,2}^*$ for $\beta=2$, $n\leq 6$.
	
	\medskip
	
		\noindent{\bf Example 3.} Punctured Domain: Let $\Omega$ be a bounded domain with Lipschitz boundary such that $0\in\Omega$. We consider $\Omega^*=\Omega\backslash\{0\}$. Notice that 
		$$\delta_{\Omega^*}(x)=\min\{|x|,\delta_\Omega(x)\}\leq |x|.$$
		Then for any $\beta$ satisfying \eqref{beta-0}, it holds
		\begin{align*}
		\mu_{n+1,\beta}(\Omega^*)=\inf_{u\in C_0^\infty(\Omega^*)\backslash\{0\}} J_{n+1,\beta,B_1^*}[u]\leq  \inf_{u\in C_0^\infty(\Omega^*)}\frac{\int_{\Omega^*}|\nabla u|^2 dx}{\big(\int_{\Omega^*}|x|^{-2+\frac{n-1}{n+1}\beta}|u|^{2+\frac{2\beta}{n+1}}dx\big)^{\frac{n+1}{n+\beta+1}}}.
		\end{align*}
		For any $\varepsilon>0$, there is $w_\varepsilon\in C_0^\infty(\mathbb{R}^{n+1}_*)$ with supp $w_\varepsilon\subset R_\varepsilon \Omega^*$ for some $R_\varepsilon>0$, such that 
		\begin{align*}
		\frac{\int_{R_\varepsilon\Omega^*}|\nabla w_\varepsilon|^2 dx}{\big(\int_{R_\varepsilon\Omega^*}|x|^{-2+\frac{n-1}{n+1}\beta}|w_\varepsilon|^{2+\frac{2\beta}{n+1}}dx\big)^{\frac{n+1}{n+\beta+1}}}<\mu_{n+1,\beta}(\mathbb{R}^{n+1}_*)+\varepsilon.
		\end{align*}
		Set $u_\varepsilon(x)=w_\varepsilon(R_{\varepsilon}x)$, then $u_\varepsilon\in C_0^\infty(\Omega^*)$ and by scaling invariant property, it is easy to check that 
		\begin{align*}
		\frac{\int_{\Omega^*}|\nabla u_{\varepsilon}|^2 dx}{\big(\int_{\Omega^*}|x|^{-2+\frac{n-1}{n+1}\beta}|u_\varepsilon|^{2+\frac{2\beta}{n+1}}dx\big)^{\frac{n+1}{n+\beta+1}}}<\mu_{n+1,\beta}(\mathbb{R}^{n+1}_*)+\varepsilon,
		\end{align*}
		which implies that $\mu_{n+1,\beta}(\Omega^*)\leq \mu_{n+1,\beta}(\mathbb{R}^{n+1}_*)+\varepsilon$. Sending $\varepsilon\to 0$, we get for any $\beta$ satisfying \eqref{beta-0},
		$$\mu_{n+1,\beta}(\Omega^*)\leq \mu_{n+1,\beta}(\mathbb{R}^{n+1}_*).$$
		In particular, for $n=1$, $\mu_{2,\beta}(\Omega^*)=0$. Besides, for $n\geq 2$, since $W_0^{1,2}(\Omega^*)=W_0^{1,2}(\Omega)$ and $\delta_{\Omega^*}\leq \delta_{\Omega}$, we can obtain that for $0\leq \beta\leq \frac{2(n+1)}{n-1}$, 
		$$\mu_{n+1,\beta}(\Omega^*)\leq\mu_{n+1,\beta}(\Omega).$$
		Then 
	\begin{align}\label{*-1}\mu_{n+1,\beta}(\Omega^*)\leq \min\{\mu_{n+1,\beta}(\mathbb{R}^{n+1}_*),  \mu_{n+1,\beta}(\Omega)\}.
	\end{align}
		
		On the other hand, notice that %$W^{1,2}_0(\Omega^*)\subset W^{1,2}_0(\Omega)\cap \cal{D}^{1,2}_0(\mathbb{R}^{n+1}_*)$ and 
		$\delta_\Omega^*(x)=\min\{\delta_{\Omega}(x),\delta_{\mathbb{R}^{n+1}_*}(x)\}$ in $\Omega^*$.  Set $\theta=\frac{\mu_{n+1,\beta}(\mathbb{R}^{n+1}_*)}{\mu_{n+1,\beta}(\mathbb{R}^{n+1}_*)+\mu_{n+1,\beta}(\Omega)}$.Then for $\beta$ satisfying \eqref{beta-p} and any $u\in C_0^\infty(\Omega^*)$, it holds
	\begin{align*}
	    \int_{\Omega^*}|\nabla u|^2dx=&\theta\int_{\Omega}|\nabla u|^2dx+(1-\theta)\int_{\mathbb{R}^{n+1}_*}|\nabla u|^2dx\\
	    \geq &\theta \mu_{n+1,\beta}(\Omega)\big(\int_{\Omega}\delta_{\Omega}^{-2+\frac{n-1}{n+1}\beta} |u|^{2+\frac{2\beta}{n+1}}dx\big)^{\frac{n+1}{n+\beta+1}}\\
	    &+(1-\theta)\mu_{n+1,\beta}(\mathbb{R}^{n+1}_*)\big(\int_{\mathbb{R}^{n+1}_*}\delta_{\mathbb{R}^{n+1}_*}^{-2+\frac{n-1}{n+1}\beta} |u|^{2+\frac{2\beta}{n+1}}dx\big)^{\frac{n+1}{n+\beta+1}}\\
	    \geq & \frac{\mu_{n+1,\beta}(\Omega)\mu_{n+1,\beta}(\mathbb{R}^{n+1}_*)}{\mu_{n+1,\beta}(\Omega)+\mu_{n+1,\beta}(\mathbb{R}^{n+1}_*)}\big(\int_{\Omega^*}\delta_{\Omega^*}^{-2+\frac{n-1}{n+1}\beta} |u|^{2+\frac{2\beta}{n+1}}dx\big)^{\frac{n+1}{n+\beta+1}}.
	\end{align*}
	It follows that 
	\begin{align}\label{*-2}
	    \mu_{n+1,\beta}(\Omega^*)\geq \frac{\mu_{n+1,\beta}(\Omega)\mu_{n+1,\beta}(\mathbb{R}^{n+1}_*)}{\mu_{n+1,\beta}(\Omega)+\mu_{n+1,\beta}(\mathbb{R}^{n+1}_*)}.
	\end{align}
	Combing \eqref{*-1}  and \eqref{*-2}, we have 
	\begin{align}\label{*-3}
	     \frac{\mu_{n+1,\beta}(\Omega)\mu_{n+1,\beta}(\mathbb{R}^{n+1}_*)}{\mu_{n+1,\beta}(\Omega)+\mu_{n+1,\beta}(\mathbb{R}^{n+1}_*)}\leq \mu_{n+1,\beta}(\Omega^*)\leq \min\{\mu_{n+1,\beta}(\mathbb{R}^{n+1}_*),  \mu_{n+1,\beta}(\Omega)\}.
	\end{align}

		\medskip
		
	\noindent{\bf Example 4.} Annular Domain: Let $\Omega_1, \Omega_2 $ be two bounded domains with Lipschitz boundary in $\mathbb{R}^{n+1}$, such that $\Omega_1\subset\subset\Omega_2$, and set $\Omega_0=\mathbb{R}^{n+1}\backslash \overline{\Omega_1}$.  Consider the domain $\Omega=\Omega_0\cap \Omega_2$.  
	%Set $\theta=\frac{\mu_{n+1,\beta}(\Omega_2)}{\mu_{n+1,\beta}(\Omega_0)+\mu_{n+1,\beta}(\Omega_2)}$. Since $W^{1,2}_0(\Omega)\subset W^{1,2}_0(\Omega_0)\cap W^{1,2}_0(\Omega_2)$ and $\delta_\Omega(x)=\min\{\delta_{\Omega_0}(x),\delta_{\Omega_2}(x)\}$ in $\Omega$, we have that, for $\beta$ satisfying \eqref{beta-p} and any $u\in W^{1,2}_0(\Omega)$,
	%\begin{align*}
	%    \int_{\Omega}|\nabla u|^2dx=&\theta\int_{\Omega_0}|\nabla u|^2dx+(1-\theta)\int_{\Omega_2}|\nabla u|^2dx\\
	%    \geq &\theta \mu_{n+1,\beta}(\Omega_0)\big(\int_{\Omega_0}\delta_{\Omega_0}^{-2+\frac{n-1}{n+1}\beta} |u|^{2+\frac{2\beta}{n+1}}dx\big)^{\frac{n+1}{n+\beta+1}}\\
	%    &+(1-\theta)\mu_{n+1,\beta}(\Omega_2)\big(\int_{\Omega_2}\delta_{\Omega_2}^{-2+\frac{n-1}{n+1}\beta} |u|^{2+\frac{2\beta}{n+1}}dx\big)^{\frac{n+1}{n+\beta+1}}\\
	%    \geq & \frac{\mu_{n+1,\beta}(\Omega_0)\mu_{n+1,\beta}(\Omega_2)}{\mu_{n+1,\beta}(\Omega_0)+\mu_{n+1,\beta}(\Omega_2)}\big(\int_{\Omega}\delta_{\Omega}^{-2+\frac{n-1}{n+1}\beta} |u|^{2+\frac{2\beta}{n+1}}dx\big)^{\frac{n+1}{n+\beta+1}}.
	%\end{align*}
	%It follows that 
	Similarly with \eqref{*-2}, it holds
	\begin{align*}
	    \mu_{n+1,\beta}(\Omega)\geq \frac{\mu_{n+1,\beta}(\Omega_0)\mu_{n+1,\beta}(\Omega_2)}{\mu_{n+1,\beta}(\Omega_0)+\mu_{n+1,\beta}(\Omega_2)}.
	\end{align*}
	
	On the other hand, assume $0\in\Omega_1$. Notice that $\{(\frac1k \Omega_0)\cap\Omega_2\}$ is a normal approximation sequence for $\Omega_2^*$ as $k \to +\infty$, then by Lemma \ref{limit} and Example 3, it holds
	\begin{align*}
	    \limsup_{k\to +\infty} \mu_{n+1,\beta}((\frac1k \Omega_0)\cap\Omega_2)
	    \leq \mu_{n+1,\beta}(\Omega_2^*)\leq \min\{\mu_{n+1,\beta}(\mathbb{R}^{n+1}_*),  \mu_{n+1,\beta}(\Omega)\}.
	    \end{align*}
	    In particular, for $n=1$,  $\lim_{k\to+\infty}\mu_{2,\beta}((\frac1k \Omega_0)\cap\Omega_2)=0$, which implies that we can find some $2$-dimensional annular domains with arbitrarily small sharp constant $\mu_{2,\beta}$. Besides, for $n\geq 2$, since $\mu_{n+1,1}(\mathbb{R}^{n+1}_*)<\mu_{n+1,1}^*$ for $\beta=1$, $n\leq 3$, and  $\mu_{n+1,2}(\mathbb{R}^{n+1}_*)<\mu_{n+1,2}^*$ for $\beta=2$, $n\leq 6$, in such cases, we can find some annular domains whose sharp constant $\mu_{n+1,\beta}$ is strictly less than $\mu_{n+1,\beta}^*$, and hereby the sharp constant can be achieved. As a special case, we consider $B_k(0)\backslash B_1(0) (k>1)$. By Lemma \ref{limit} and \eqref{HS-outball}, it holds for $\beta$ satisfying \eqref{beta-strict},
	   \begin{align*}
	   \limsup_{k \to +\infty}\mu_{n+1,\beta}(B_k(0)\backslash B_1(0))\leq\mu_{n+1,\beta}(B_1^c(0))<\mu_{n+1,\beta}^*.
	   \end{align*}
	   Then for any $n\geq 2$ and $\beta$ satisfying \eqref{beta-strict}, we can find $B_{k_n}(0)\backslash B_1(0)$ for some $k_n>1$,  whose sharp constant $\mu_{n+1,\beta}$ is strictly less than $\mu_{n+1,\beta}^*$.

\vskip 1cm
\noindent {\bf Acknowledgements}\\
%The project is supported by  the National Natural Science Foundation of China (Grant No. 12071269). 
Zhu is partially supported by the Simons collaboration grant. Wang is supported by China
Postdoctoral Science Foundation.
 
 %% bibliography--------------------------------------------------------------------
%\begin{center}

\end{document}